\documentclass[final,hidelinks,onefignum,onetabnum]{siamart220329}

\usepackage{amscd,amssymb, amsmath,amsfonts, wasysym, mathrsfs, enumerate, mathtools,hhline,xcolor}%
\usepackage{graphicx}
\usepackage[all, cmtip]{xy}
\usepackage{multirow}

\definecolor{hot}{RGB}{65,105,225}
\usepackage{algorithm}
\usepackage[noend]{algpseudocode}

\usepackage{hyperref}
\hypersetup{
	plainpages=false,
    colorlinks,
    linktocpage=true,
    linkcolor=hot,
    citecolor=hot,
    urlcolor=hot
}
\usepackage[hyperpageref]{backref}

\usepackage[colorinlistoftodos,prependcaption,textsize=footnotesize]{todonotes}

\makeatletter
\@mparswitchfalse%
\makeatother
\normalmarginpar %

\newtheorem{remark}[theorem]{Remark}

\begin{document}

\title{A numerical domain decomposition method for solving elliptic equations on manifolds}

\author{Shuhao Cao \thanks{Division of Computing, Analytics, and Mathematics, School of Science and Engineering, University of Missouri-Kansas City, Kansas City, MO
  (\email{scao@umkc.edu}).}
\and Lizhen Qin \thanks{Mathematics Department, Nanjing University, Nanjing, Jiangsu, China
  (\email{qinlz@nju.edu.cn}).}}

\maketitle

\headers{DDM on Manifolds}{S. Cao and L. Qin}

\begin{abstract}
A new numerical domain decomposition method is proposed for solving elliptic equations on compact Riemannian manifolds. The advantage of this method is to avoid global triangulations or grids on manifolds. Our method is numerically tested on some $4$-dimensional manifolds such as the unit sphere $S^{4}$, the complex projective space $\mathbb{CP}^{2}$ and the product manifold $S^{2} \times S^{2}$.
\end{abstract}

\begin{AMS}{Primary 65N30; Secondary 58J05, 65N55.}
\end{AMS}


\section{Introduction}\label{sec_introduction}
Elliptic partial differential equations on Riemannian manifolds are of fundamental importance both in analysis and geometry (see e.g., \cite{schoen_yau,jost}). A simple and important example is
\begin{equation}
\label{eqn_laplace_problem}
- \Delta u + bu = f.
\end{equation}
Here $\Delta$ is the Laplace-Beltrami operator, or Laplacian for brevity, defined on a $d$-dimensional Riemannian manifold $M$. Many manifolds are naturally submanifolds of Euclidean spaces, here the ``dimension'' of a submanifold is referred to as the topological dimension of the manifold, not the dimension of its ambient Euclidean space. For example, the $n$-dimensional unit sphere $S^n$ is embedded in $\mathbb{R}^{n+1}$.

When the manifold $M$ is a two-dimensional Riemannian submanifold of $\mathbb{R}^{3}$, i.e. a surface, the numerical methods to solve PDEs, particularly \eqref{eqn_laplace_problem}, on $M$ have been extensively studied for a long history (see e.g. \cite{nedelec_planchard,nedelec,baumgardner_frederickson,dziuk88,dziuk91,dziuk_elliott,olshanskii2009finite}). Over several decades, among many others, the surface finite element method and its variant have had far-reaching developments (see e.g., \cite{antonietti2015high,bachini2021arbitrary,demlow2009higher,demlow2007adaptive,frittelli2018virtual,jankuhn2021error,olshanskii2017trace,reusken2015analysis}), and applications to various PDEs (see e.g. \cite{bachini2021intrinsic,bonito2020divergence,brannick2022bootstrap,dziuk2007finite,elliott2015evolving,jin2021gradient}); see also e.g. \cite{DDE05,dziuk_elliott,BDN20} for surveys and bibliographies, and e.g. \cite{2021MFEM,2007dealii,2000CGAL} for software developments. They have also been widely applied to various areas such as computer graphics (e.g. \cite{etzmuss2003fast}), surface fitting (e.g. \cite{Dobrev2021}), shape analysis (e.g. \cite{dobrev2010surface,reuter2010hierarchical,reuter2009discrete}), isogeometric topology optimization (e.g. \cite{juttler2016numerical,kang2016isogeometric}), and medical imaging (e.g. \cite{mohamed2005finite}). A basic and common feature of these finite element methods is to construct a global triangulation of $M$, which provides a necessary grid structure for a \emph{global} finite element space. This triangulation, following \cite[Definition~8.3]{munkres} in the sense of differential topology, is a bijection between a simplicial complex and the manifold $M$, which satisfies certain regularity. In computation, a concrete representation of this triangulation is to approximate $M$ by a polyhedron in $\mathbb{R}^{3}$. Alternatively, $M$ is represented implicitly as a level set of a function $\phi$, say $M= \phi^{-1} (0)$, and an equation on $M$ can be solved by the methods of trace elements or implicit surfaces elements (e.g. \cite{olshanskii2017trace,reusken2015analysis,dziuk_elliott}).

However, there are several essential difficulties to apply the methods above to general higher dimensional manifolds. First, many important and interesting examples of higher dimensional manifolds are not submanifolds of Euclidean spaces by definition.
A notable example is the complex projective spaces $\mathbb{CP}^{n}$, which serve as the foundation for algebraic geometry.
Though Whitney Embedding Theorem (\cite{whitney36}, see also \cite[p.~24-27]{hirsch}) and Nash Embedding Theorem (\cite{nash}) do reveal that every smooth manifold can be embedded into a Euclidean space $\mathbb{R}^{k}$ differential-topologically or even geometrically for some large $k$.
To the best of our knowledge, there is no literature on a universal algorithm to efficiently construct such an explicit embedding for computation on general higher dimensional manifolds.
Second, assuming $M$ is a submanifold of $\mathbb{R}^{k}$, if the codimension of the submanifold $M$ in $\mathbb{R}^{k}$ is greater than $1$, which is often the case, it will be horribly difficult to find an effective polytopal approximation to $M$ in general due to topological and geometrical complexity, meanwhile, $M$ also cannot be represent as a level set of a function.

A global triangulation of a manifold is helpful in numerical computation, since it can provide a global discretization of the target problem. We do acknowledge that J.~H.~C.~Whitehead proved (\cite{whitehead}, see also \cite[Theorem~10.6]{munkres}) that every smooth manifold can be globally triangulated in an abstract way. However, to the best of our knowledge, in practice, there has been no algorithm to build such a concrete triangulation or a grid over a high-dimensional manifold in general. To circumvent this difficulty, in \cite{qin_zhang_zhang} which serves as our major inspiration, Qin--Zhang--Zhang proposed a new idea to numerically solve elliptic PDEs on manifolds by avoiding global triangulations completely.
Since a $d$-dimensional manifold $M$ has local coordinate charts by definition, $M$ can be decomposed into finitely many
subdomains that carry local Cartesian coordinates. Consequently, an elliptic equation on each subdomain can be transformed to one on a domain in $\mathbb{R}^{d}$. Thus an elliptic problem on $M$ can be assembled to either coupled problems on Euclidean domains, or can be solved directly by Domain Decomposition Methods (DDMs). This idea had been numerically verified on the $3$-dimensional unit sphere $S^{3}$ in \cite{qin_zhang_zhang}, where $S^{3}$ is decomposed into two subdomains. However, a major drawback of \cite{qin_zhang_zhang} is its lack of flexibility to deal with more general manifolds, whose charts may involve more than two subdomains.

In this paper, we shall develop the idea in \cite{qin_zhang_zhang} further to solve problems on general manifolds. Similar to \cite{qin_zhang_zhang}, global triangulations or grids shall be completely avoided. In fact, we shall solve such problems by an overlapping domain decomposition method (DDM). The development of DDM has a long history. It was first invented by H.~A.~Schwarz \cite{schwarz}. The version of DDM we mainly follow was originally proposed by P.~L.~Lions in \cite[I.~4]{lions1} for solving continuous problems in Euclidean spaces. It has been well-developed and is later named as Multiplicative Schwarz Method (c.f. \cite{toselli_widlund}).
The later development usually takes this DDM as a preconditioner for a globally discretized problem (see e.g. \cite{BPWX,engquist1998absorbing},
and pure algebraic versions in e.g. \cite{1996CaiSaadOverlapping,2017LiSaadLow}).
However, a globally discretized problem should be based on a global grid on a manifold $M$, which is not accessible in our case.
Therefore, we shall more closely follow Lions' original approach rather than the later development. This original approach is a very simple iteration scheme such that a local problem in a subdomain is solved in each step. We found that this DDM can be well-adapted to solve problems on manifolds. As in \cite{qin_zhang_zhang}, a problem in each subdomain can be converted to one in a Euclidean domain. It is simpler to obtain a grid and hence a discretization over each Euclidean domain. To minimize the difficulty of coding, these Euclidean domains can be even chosen as (high-dimensional) rectangles. The transition of information among these Euclidean domains is by virtue of the transition maps of coordinates.

Overall, our approach is a discrete imitation of Lions' method with additional transition maps of coordinates. It is necessary to point out that the transition maps in the proposed method are not required to preserve nodes or grids. More precisely, in computations, suppose there are two Euclidean domains $D_i$ and $D_j$. Both $D_i$ and $D_j$ carry a grid, respectively. If $\phi_{ji}: D_{i} \rightarrow D_{j}$ is such a transition map between $D_{i}$ and $D_{j}$, the image of the grid on $D_i$ under $\phi_{ji}$ may not match the grid on $D_j$. Furthermore, for a node $\xi \in D_i$, $\phi_{ji}(\xi)$ may not be a node anymore on $D_j$ (see Fig. \ref{fig:phiij} for an example).
Nevertheless, this incompatibility shows the high flexibility of our approach. In fact, if all transition maps preserved grids, we would obtain a global grid on $M$ which is practically almost impossible. The transition maps in \cite[Section~4]{qin_zhang_zhang} do not preserve grids but do preserve boundary nodes on each subdomain, i.e., the $\phi_{ji}$ above maps nodes on $\partial D_{i}$ to nodes of $D_{j}$. As a result, the methods therein have some flexibility to solve PDEs on $S^{3}$. In order to handle more general high-dimensional manifolds, this paper improves the flexibility further. Particularly, interpolation techniques are employed now to handle nonmatching grids.

The proposed approach is numerically verified on closed manifolds of dimension $4$. They are $4$-dimensional unit sphere $S^{4}$, the complex projective space $\mathbb{CP}^{2}$, and the product manifold $S^{2} \times S^{2}$. The numerical results show that our method solves equation \eqref{eqn_laplace_problem} on these manifolds in a natural manner.

The outline of this paper is as follows. In Section \ref{sec_continuous}, we shall extend P.~L.~Lions' method in \cite{lions1} to a DDM for continuous problems on manifolds and prove its convergence. In Section \ref{sec_implement}, we shall propose our discrete imitation of Lions' method. Some specialty of product manifolds will be explained in Section \ref{sec_product}. Finally, some numerical results will be presented in Section \ref{sec_experiment}.

\section{Theory on Continuous Problems}\label{sec_continuous}
In this section, we shall first formulate a second order elliptic model problem on general manifolds. Then we introduce a domain decomposition method (DDM) generalized from \cite[I.~4]{lions1} to solve the model problem (see Algorithm \ref{alg_continuous} below).
We shall formulate and prove Theorem \ref{thm_convergence} below on the convergence of this iterative procedure. This procedure also motivates us to propose Algorithm \ref{alg_discrete} in next section, which gives numerical approximations to the solution to the model problem.

Let $M$ be a $d$-dimensional compact smooth manifold without or with boundary $\partial M$. Equipping $M$ with a Riemannian metric $g$, the Laplacian $\Delta$
can be defined on $M$. In general, neither $g$ nor $\Delta$ can be expressed by coordinates globally because $M$ does not necessarily have a global coordinate chart. In a local chart with coordinates $(x_{1}, \dots, x_{d})$, the Riemannian metric tensor $g$ is expressed as
\begin{equation}
\label{eqn_metric}
g= \sum_{\alpha, \beta=1}^{d} g_{\alpha \beta} \mathrm{d} x_{\alpha} \otimes \mathrm{d} x_{\beta},
\end{equation}
where the matrix $(g_{\alpha \beta})_{d \times d}$ is symmetric and positive definite. The Laplacian $\Delta$ can be then expressed in this chart as
\[
\Delta u = \frac{1}{\sqrt{G}}\sum_{\alpha=1}^{d}\frac{\partial}{\partial x_{\alpha}} \left( \sum_{\beta=1}^{d}g^{\alpha \beta} \sqrt{G} \frac{\partial u}{\partial x_{\beta}} \right),
\]
where $G = \det \left( (g_{\alpha \beta})_{d \times d} \right)$ is the determinant of the matrix $(g_{\alpha \beta})_{d \times d}$ and $(g^{\alpha \beta})_{d \times d}$ is the inverse of $(g_{\alpha \beta})_{d \times d}$. It is well-known that $\Delta$ is an elliptic differential operator of second order.

We consider the following model problem on $M$
\begin{equation}\label{eqn_problem}
\left\{
\begin{aligned}
- \Delta u + bu & = f,
\\
u|_{\partial M} & = 0,
\end{aligned}
\right.
\end{equation}
where $b \geq 0$ is a constant and $f \in L^{2} (M)$. A weak solution to \eqref{eqn_problem} is a solution to the following problem: Find a $u \in H^{1}_{0} (M)$ such that, $\forall v \in H^{1}_{0} (M)$,
\begin{equation}\label{eqn_problem_weak}
\int_{M} (\langle \nabla u, \nabla v \rangle + buv)\, \mathrm{dvol} = \int_{M} fv \,\mathrm{dvol}.
\end{equation}
Here $\nabla u$ and $\nabla v$ are the gradients of $u$ and $v$ with respect to $g$, $\langle \nabla u, \nabla v \rangle$
is the inner product of $\nabla u$ and $\nabla v$, and $\mathrm{dvol}$ is the volume form. In terms of local coordinates,
\begin{equation}\label{eqn_gradient_product}
\langle \nabla  u, \nabla  v \rangle = \sum_{\alpha,\beta=1}^{d} g^{\alpha \beta} \frac{\partial u}{\partial x_{\alpha}} \frac{\partial v}{\partial x_{\beta}}
\end{equation}
and
\begin{equation}\label{eqn_volume}
\mathrm{dvol} = \sqrt{G} \mathrm{d} x_{1} \cdots \mathrm{d} x_{d}.
\end{equation}
For brevity, we shall omit the symbol $\mathrm{dvol}$ in integrals on $M$ throughout this paper.

Note that it suffices to solve \eqref{eqn_problem_weak} on each component of $M$. Therefore, without loss of generality, $M$ is assumed to be connected. In addition, if $\partial M = \emptyset$, the condition above $u|_{\partial M} =0$ is vacuously satisfied in regards to $H^{1}_{0} (M) = H^{1} (M)$. To guarantee \eqref{eqn_problem_weak} is well-posed, we further assume $b>0$ if $\partial M = \emptyset$. (Actually, if $b=0$, one may imposed additional conditions such as $\int_{M} u \mathrm{dvol} =0$ to guarantee the well-posedness. On the other hand, the numerical algorithm would be more complicated in this situation. We shall study the case later.)

Now we describe a domain decomposition iterative procedure to solve \eqref{eqn_problem_weak}.
This method was originally proposed by P.~L.~Lions in \cite[I.~4]{lions1}, in which the classical Schwarz Alternating Method is extended to an iterative procedure with many subdomains.
The manifold nature of this method is intrinsically adapted to solve PDEs on manifolds. More precisely, suppose $M$ is decomposed into $m$ subdomains, i.e.
\[
M = \bigcup_{i=1}^{m} \mathrm{Int} M_{i}.
\]
Here $M_{i}$ is a closed subdomain (submanifold with codimension $0$) of $M$ with Lipschitz boundary, and $\mathrm{Int} M_{i}$ is the interior of $M_{i}$ in the sense of point-set topology of $M$. Clearly, an element in $H^{1}_{0} (M_{i})$ can be naturally considered as an element in $H^{1}_{0} (M)$ by zero extension. Thus, $H^{1}_{0} (M_{i}) \subseteq H^{1}_{0} (M)$. The iterative procedure to solve \eqref{eqn_problem_weak} is the following Algorithm \ref{alg_continuous}.

\medskip
\begin{algorithm}[ht]
\caption{A DDM for the continuous problem.}
\label{alg_continuous}

\begin{algorithmic}[1]
\State%
Choose an arbitrary initial guess $u^{0} \in H^{1}_{0} (M)$ for \eqref{eqn_problem_weak}.

\State%
For each $n>0$, assuming $u^{n-1}$ has been obtained, define $u^{n,0} = u^{n-1}$. For $1 \leq i \leq m$, assuming $u^{n,j}$ has been obtained for all $j<i$, find a $u^{n,i} \in H^{1}_{0} (M)$ such that
\begin{equation}\label{alg_continuous_1}
\left\{
\begin{array}{rcl}
(\forall v \in H^{1}_{0} (M_{i})) \ \ \ \int_{M_{i}} \langle \nabla u^{n,i}, \nabla v \rangle + bu^{n,i} v & = & \int_{M_{i}} fv, \\
u^{n,i}|_{M \setminus \mathrm{Int} M_{i}} & = & u^{n,i-1}|_{M \setminus \mathrm{Int} M_{i}}.
\end{array}
\right.
\end{equation}

\State%
Let $u^{n} = u^{n,m}$.

\end{algorithmic}
\end{algorithm}

\medskip
To obtain the $u^{n,i}$ in \eqref{alg_continuous_1}, one first solve an elliptic problem on $M_{i}$ with Dirichlet boundary condition $u^{n,i}|_{\partial M_{i}} = u^{n,i-1}|_{\partial M_{i}}$, then extend the solution to a function $u^{n,i}$ on $M$ by defining $u^{n,i}|_{M \setminus M_{i}} = u^{n,i-1}|_{M \setminus M_{i}}$. Thus the $u^{n,i}$ is well-defined and hence Algorithm \ref{alg_continuous} is well-posed.

We have the following theorem on the geometrical convergence of Algorithm \ref{alg_continuous}.
\begin{theorem}\label{thm_convergence}
There exist constants $C_{0} >0$ and $L \in [0,1)$ such that, $\forall u^{0} \in H^{1}_{0} (M)$, $\forall n>0$,
\[
\| u- u^{n} \|_{H^{1}_{0} (M)} \leq C_{0} L^{n} \| u- u^{0} \|_{H^{1}_{0} (M)},
\]
where $u$ is the solution to \eqref{eqn_problem_weak} and $u^{n}$ is the $n$th iterated approximation in Algorithm \ref{alg_continuous} with initial guess $u^{0}$.
\end{theorem}

Theorem \ref{thm_convergence} was originally proved by P.~L.~Lions (\cite[Theorem.~I.2]{lions1}) in the case that $M$ is a Euclidean domain. We shall adapt his proof to the case of manifolds.

An elementary proof of the following lemma can be found in \cite[p.~17]{lions1}. It is also an immediate corollary of \cite[(1.2)]{xu_zikatanov}.
\begin{lemma}\label{lem_projection}
Suppose $V$ is a Hilbert space and $V_{i}$ ($1 \leq i \leq m$) are closed subspaces of $V$ such that $V = \sum_{i=1}^{m} V_{i}$. Then
\[
\| P_{V_{m}^{\bot}} P_{V_{m-1}^{\bot}} \cdots P_{V_{1}^{\bot}} \| < 1,
\]
where each $V_{i}^{\bot}$ is the orthogonal complement of $V_{i}$, and $P_{V_{i}^{\bot}}$ is the orthogonal projection of $V$ onto $V_{i}^{\bot}$.
\end{lemma}

Define the energy bilinear form on $H^{1}_{0} (M)$ as
\[
a(w,v) = \int_{M} \langle \nabla w, \nabla v \rangle + bwv.
\]

\begin{proof}[Proof of Theorem \ref{thm_convergence}]
It's easy to see that the energy norm $a(\cdot, \cdot)^{\frac{1}{2}}$ is equivalent to the original $H_{1}$-norm on $H^{1}_{0} (M)$. In other words, there are positive constants $C_{1}$ and $C_{2}$ such that $C_{1} a(\cdot, \cdot)^{\frac{1}{2}} \leq \| \cdot \|_{H^{1}_{0} (M)} \leq C_{2} a(\cdot, \cdot)^{\frac{1}{2}}$. Thus it suffices to show that: there exists $L \in [0,1)$ such that, $\forall n>0$,
\begin{equation}\label{thm_convergence_1}
a(u-u^{n}, u-u^{n})^{\frac{1}{2}} \leq L^{n} a(u-u^{0}, u-u^{0})^{\frac{1}{2}}.
\end{equation}
Actually, we will have $\| u- u^{n} \|_{H^{1}_{0} (M)} \leq C_{0} L^{n} \| u- u^{0} \|_{H^{1}_{0} (M)}$ with $C_{0} = C_{1}^{-1} C_{2}$. (See also \cite[Lemma~4.4]{qin_xu} and its proof.) Let $V_{i}$ denote $H^{1}_{0} (M_{i})$. By adapting the argument in \cite[Theorem.~I.2]{lions1}, we obtain
\[
u- u^{n,i} = P_{V_{i}^{\bot}} (u- u^{n,i-1}),
\]
where $P_{V_{i}^{\bot}}$ is the orthogonal projection onto $V_{i}^{\bot}$ with respect to $a(\cdot, \cdot)$. Therefore
\begin{eqnarray}\label{thm_convergence_2}
u - u^{n} & = & u - u^{n,m} = P_{V_{m}^{\bot}} (u- u^{n,m-1}) = \cdots = P_{V_{m}^{\bot}} \cdots P_{V_{1}^{\bot}} (u- u^{n,0}) \nonumber \\
& = & P_{V_{m}^{\bot}} \cdots P_{V_{1}^{\bot}} (u- u^{n-1}) = (P_{V_{m}^{\bot}} \cdots P_{V_{1}^{\bot}})^{n} (u- u^{0}).
\end{eqnarray}
By Lemma \ref{lem_projection}, we have
\[
\| P_{V_{m}^{\bot}} \cdots P_{V_{1}^{\bot}} \| < 1
\]
with respect to $a(\cdot, \cdot)$. Define $L = \| P_{V_{m}^{\bot}} \cdots P_{V_{1}^{\bot}} \|$, then $L \in [0,1)$ and \eqref{thm_convergence_2} implies \eqref{thm_convergence_1} which finishes the proof.
\end{proof}

\section{Numerical Scheme}
\label{sec_implement}
In this section, we propose a numerical DDM iterative procedure (Algorithm \ref{alg_discrete} below) to obtain approximations to the solution of \eqref{eqn_problem_weak}. The procedure is as follow. First, $M$ is decomposed into overlapping subdomains $M_{i}$ ($1 \leq i \leq m$), and each $M_{i}$ is in a coordinate chart. Second, a DDM iterative procedure, which serves as a discrete counterpart of Algorithm \ref{alg_continuous} is applied. Due to $M_{i}$ being a coordinate chart, an elliptic problem on $M_{i}$ can be naturally converted to one on a domain in a Euclidean space. Then, this problem on $M_{i}$ can be solved approximately using conventional finite element methods. The transition of information among subdomains is by interpolation.

For the purpose of presentation, only manifolds without boundary are considered in the numerical examples.
A forthcoming work will study in detail the numerical implementation on manifolds with boundaries. In that more general case, we shall have to apply some special technique to deal with the boundary.

\subsection{Finite Element Spaces over a \texorpdfstring{d}{$d$}-Rectangle}
Suppose a manifold $M$ has dimension $d$.
As indicated above, each subdomain $M_{i}$ of $M$ shall be converted
to a domain $D_{i} \subset \mathbb{R}^{d}$. This conversion substantially reduces the difficulty of the numerical scheme in consideration. However, when $d>3$, the construction of a discretized problem on $D_{i}$ still remains a difficult task. The main reason is that the geometric intuition used in implementing finite element spaces in $\mathbb{R}^2$ or $\mathbb{R}^3$ cannot be simply ported to higher dimensions.
To minimize the difficulty of tessellation in higher dimensions, we shall choose $D_{i}$ as a $d$-rectangle and use the tensor product-type finite element space of $d$-rectangles (see e.g., \cite[p.~56-64]{ciarlet}).

Recall that a $d$-rectangle $D$ is
\[
D = \prod_{i=1}^{d} [a_{i}, b_{i}] = \{ (x_{1}, \cdots, x_{d}) \mid \forall i, x_{i} \in [a_{i}, b_{i}] \}.
\]
We refine each coordinate factor interval $[a_{i}, b_{i}]$ by adding points of partition:
\[
a_{i} = c_{i,0} < c_{i,1} < \cdots < c_{i,N_{i}} = b_{i}.
\]
Then $[a_{i}, b_{i}]$ is divided into $N_{i}$ subintervals. Define a function $\varphi_{i,j}$ on $[a_{i}, b_{i}]$ for $0 \leq j \leq N_{i}$ as
\begin{equation}\label{eqn_1d_base}
\varphi_{i,j} (x_{i}) =
\begin{cases}
\frac{x_{i} - c_{i,j-1}}{c_{i,j} - c_{i,j-1}}, & x_{i} \in [c_{i,j-1}, c_{i,j}]; \\
\frac{x_{i} - c_{i,j+1}}{c_{i,j} - c_{i,j+1}}, & x_{i} \in [c_{i,j}, c_{i,j+1}]; \\
0, & \text{otherwise}.
\end{cases}
\end{equation}
Here $\varphi_{i,j} (x_{i})$ is undefined for $x_{i} < c_{j}$ (resp. $x_{i} > c_{j}$) when $j=0$ (resp. $j= N_{i}$). Clearly, $\varphi_{i,j}$ is piecewise linear such that $\varphi_{i,j} (c_{i,j}) =1$ and $\varphi_{i,j} (c_{i,t}) =0$ for $t \neq j$.

The refinement of all such $[a_{i}, b_{i}]$ provides a grid on $D$. This divides $D$ into $\prod_{i=1}^{d} N_{i}$ many small $d$-rectangles
\begin{equation}\label{eqn_element}
\prod_{i=1}^{d} [c_{i,t_{i}-1}, c_{i,t_{i}}],
\end{equation}
where $1 \leq t_{i} \leq N_{i}$ for all $i$. Each small $d$-rectangle \eqref{eqn_element} is an element of the grid. A vertex of \eqref{eqn_element} is a node of the grid which is of the form
\[
\xi = (c_{1, j_{1}}, c_{2, j_{2}}, \cdots, c_{d, j_{d}}),
\]
where $0 \leq j_{i} \leq N_{i}$ for all $i$. Let $W_{h}$ be the finite element space of $d$-rectangles of type ($1$) (see \cite[p.~57]{ciarlet}). A base function in $W_{h}$ associated with the node $\xi$ is
\[
\varphi_{\xi} (x_{1}, \cdots, x_{d}) = \prod_{i=1}^{d} \varphi_{i,j_{i}} (x_{i}),
\]
where $\varphi_{i,j_{i}}$ is the one in \eqref{eqn_1d_base}.%

Since $D$ is a $d$-rectangle, it is relatively easy to handle the finite element space of $d$-rectangles. This advantage had been indicated in \cite[p.~62]{ciarlet} %

\subsection{Discrete Iterative Procedure}
Let $M$ be a $d$-dimensional compact Riemannian manifold without boundary. We try to find numerical approximations to the solution to the problem \eqref{eqn_problem_weak} on $M$.

Suppose $M = \bigcup_{i=1}^{m} \mathrm{Int} M_{i}$, and there is a smooth diffeomorphism $\phi_{i}: D_{i} \rightarrow M_{i}$ for each $i$, where $D_{i}$ is a $d$-rectangle in $\mathbb{R}^{d}$. Theoretically, we can always get such triples $(M_{i}, D_{i}, \phi_{i})$. Actually, for each $\zeta \in M$, there is an open chart neighborhood $U_{\zeta}$ of $\zeta$, i.e. there is a diffeomorphism $\phi_{\zeta}: \Omega_{\zeta} \rightarrow U_{\zeta}$, where $\Omega_{\zeta}$ is an open subset of $\mathbb{R}^{d}$. Since $\phi_{\zeta}^{-1} (\zeta)$ is an interior point of $\Omega_{\zeta}$, we can choose a rectangular neighborhood $D_{\zeta}$ of $\phi_{\zeta}^{-1} (\zeta)$ such that $D_{\zeta} \subseteq \Omega_{\zeta}$. This yields a diffeomorphism $\phi_{\zeta}: D_{\zeta} \rightarrow M_{\zeta} \subset U_{\zeta}$, where $M_{\zeta}$ is a neighborhood of $\zeta$. The interiors of all such $M_{\zeta}$ provide an open covering of $M$. Since it is compact, $M$ can be covered by the interiors of finitely many such $M_{\zeta}$. These finitely many $(M_\zeta, D_\zeta, \phi_{\zeta})$ yield a desired decomposition of $M$.

Let $(x_{1}, \dots, x_{d})$ denote the coordinates on $D_{i}$. Then the Riemannian metric $g$ on $M_{i}$ can be expressed as \eqref{eqn_metric}. Define an energy bilinear form on $H^{1} (D_{i})$ as
\begin{equation}\label{eqn_energy}
a_{i} (w,v) = \int_{D_{i}} \left( \sum_{\alpha,\beta=1}^{d} g^{\alpha \beta} \frac{\partial w}{\partial x_{\alpha}} \frac{\partial v}{\partial x_{\beta}} + bwv \right) \sqrt{G} \mathrm{d} x_{1} \cdots \mathrm{d} x_{d}.
\end{equation}
Define a bilinear form $(\cdot, \cdot)_{i}$ on $L^{2} (D_{i})$ as
\[
(w,v)_{i} = \int_{D_{i}} wv \sqrt{G} \mathrm{d} x_{1} \cdots \mathrm{d} x_{d}.
\]
By \eqref{eqn_gradient_product} and \eqref{eqn_volume}, the first line of \eqref{alg_continuous_1} is converted to the following equation: $\forall v \in H^{1}_{0} (D_{i})$,
\[
a_{i} (u^{n,i} \circ \phi_{i}, v) = (f \circ \phi_{i}, v)_{i}.
\]
Create a grid of $d$-rectangles over $D_{i}$. Let $V_{h}^{i}$ be the finite element space of $d$-rectangles of type $(1)$ over $D_{i}$. A discrete imitation of the first line of \eqref{alg_continuous_1} would be: find a $u_{h}^{n,i} \in V_{h}^{i}$ such that, $\forall v_{h} \in V_{h}^{i} \cap H^{1}_{0} (D_{i})$,
\[
a_{i} (u_{h}^{n,i}, v_{h}) = (f \circ \phi_{i}, v_{h})_{i}.
\]
However, this discrete problem is not well-posed because the degrees of freedom of $u_{h}^{n,i}$ on $\partial D_{i}$ are undetermined. As an imitation of the second line of \eqref{alg_continuous_1}, we should evaluate these degrees of freedom by the data in $d$-rectangles $D_{j}$ for $j \neq i$. So we have to investigate the transitions of coordinates.

For $i \neq j$, let $D_{ij} = \phi_{i}^{-1} (M_{i} \cap M_{j}) \subseteq D_{i}$ and $D_{ji} = \phi_{j}^{-1} (M_{i} \cap M_{j}) \subseteq D_{j}$ (see Fig. \ref{fig:phiij} for an illustration). Then
\[
\phi_{j}^{-1} \circ \phi_{i}: \ D_{ij} \rightarrow D_{ji}
\]
is a diffeomorphism which is the transition of coordinates on the overlapping between $M_{i}$ and $M_{j}$. As pointed out in Section \ref{sec_introduction}, $\phi_{j}^{-1} \circ \phi_{i}$ preserves neither nodes nor grid necessarily.
In other words, $\phi_{j}^{-1} \circ \phi_{i}$ may neither map a node in $D_{ij}$ to a node in $D_{ji}$, nor map the grid over $D_{ij}$ to the one over $D_{ji}$. Since $\phi_{i}$ and $\phi_{j}$ can be quite arbitrary, $\phi_{j}^{-1} \circ \phi_{i}$ may map the grid over $D_{ij}$ to intractable curves in $D_{ji}$.

\begin{figure}[htbp]
\centering
  \includegraphics[width=0.7\textwidth]{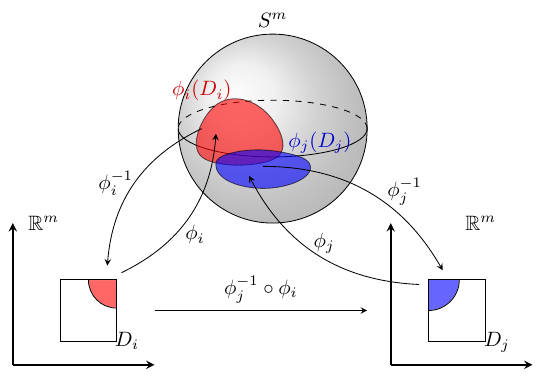}
  \caption{An illustration of a transition of coordinates on $S^m$.}
  \label{fig:phiij}
\end{figure}

However, this incompatibility among the grids over different $D_{i}$ is not a bad sign of our method. We wish to emphasize that this actually shows the high flexibility of our approach. In fact, if all the $\phi_{j}^{-1} \circ \phi_{i}$ preserved the grids, one would obtain a global grid on $M$. As mentioned in Section \ref{sec_introduction}, it is too difficult to obtain such a grid in practice. In \cite[Section~4]{qin_zhang_zhang}, the transitions of coordinates do not preserve grid but do preserve boundary nodes, i.e. the transitions map nodes on $D_{ij} \cap \partial D_{i}$ to nodes of $D_{ji}$. As a result, the method in \cite{qin_zhang_zhang} have some flexibility. The problem \eqref{eqn_problem_weak} on $S^{3}$ was solved numerically by three ways in \cite{qin_zhang_zhang}. However, those methods are not flexible enough to solve problems on more complicated manifolds in practice.
The method in this paper improves the work in \cite{qin_zhang_zhang}. Since our transitions of coordinates do not necessarily preserve nodes, we shall evaluate the degrees of freedom on $\partial D_{i}$ by interpolation.

Now we are in a position to propose our discrete Algorithm \ref{alg_discrete}. Define
\[
V_{h} = \bigoplus_{i=1}^{m} V_{h}^{i}.
\]

\medskip

\begin{algorithm}[ht]
\caption{A numerical DDM with a known chart decomposition.}
\label{alg_discrete}
\begin{algorithmic}[1]
\State Choose an arbitrary initial guess $u_{h}^{0} = (u_{h}^{0,1}, \cdots, u_{h}^{0,m}) \in V_{h}$.

\State For each $n>0$ and $1 \leq i \leq m$, assuming $u_{h}^{n-1}$ and $u_{h}^{n,j}$ have been obtained for all $j<i$, find a $u_{h}^{n,i} \in V_{h}^{i}$ as follows. Suppose $\xi \in \partial D_{i}$ is a node. If $\{ j \mid j<i, \phi_{i} (\xi) \in M_{j} \} \neq \emptyset$, then let $j_{0}$ be the maximum of this set and define
\[
u_{h}^{n,i} (\xi) = u_{h}^{n,j_{0}} (\phi_{j_{0}}^{-1} \circ \phi_{i} (\xi)).
\]
Otherwise, let $j_{0} = \max \{ j \mid j \neq i, \phi_{i} (\xi) \in M_{j} \}$ and define
\[
u_{h}^{n,i} (\xi) = u_{h}^{n-1,j_{0}} (\phi_{j_{0}}^{-1} \circ \phi_{i} (\xi)).
\]
The interior degrees of freedom of $u_{h}^{n,i}$ are determined by $\forall v_{h} \in V_{h}^{i} \cap H^{1}_{0} (D_{i})$,
\[
a_{i} (u_{h}^{n,i}, v_{h}) = (f \circ \phi_{i}, v_{h})_{i}.
\]

\State Define $u_{h}^{n} = (u_{h}^{n,1}, \cdots, u_{h}^{n,m}) \in V_{h}$.

\end{algorithmic}
\end{algorithm}

\medskip
Note that, in the second step of Algorithm \ref{alg_discrete}, it is possible that
\[
\{ j \mid j<i, \phi_{i} (\xi) \in M_{j} \} = \emptyset,
\]
for instance, $i=1$. However, the following always holds
\[
\{ j \mid j \neq i, \phi_{i} (\xi) \in M_{j} \} \neq \emptyset
\]
because $M= \bigcup_{j=1}^{m} \mathrm{Int} M_{j}$ and $\phi_{i} (\xi) \notin \mathrm{Int} M_{i}$. The choice of $u_{h}^{n,j_{0}}$ or $u_{h}^{n-1,j_{0}}$ follows the principle that we use the latest iterate in other subdomains to evaluate the boundary value of $u_{h}^{n,i}$. Also note that $\phi_{j_{0}}^{-1} \circ \phi_{i} (\xi)$ is not necessarily a node. However, we can calculate $u_{h}^{n,j_{0}} (\phi_{j_{0}}^{-1} \circ \phi_{i} (\xi))$ or $u_{h}^{n-1,j_{0}} (\phi_{j_{0}}^{-1} \circ \phi_{i} (\xi))$ by virtue of the coordinates of $\phi_{j_{0}}^{-1} \circ \phi_{i} (\xi)$ in $D_{j_{0}}$. This is essentially by an interpolation of the degrees of freedom of $u_{h}^{n,j_{0}}$ or $u_{h}^{n-1,j_{0}}$.

Now the $u_{h}^{n}$ in Algorithm \ref{alg_discrete} is the $n$th iterated discrete approximation to the solution to \eqref{eqn_problem_weak}. Unlike the $u^{n,i}$ in Algorithm \ref{alg_continuous} which is globally defined on $M$, the $u_{h}^{n,i}$ is a component of $u_{h}^{n}$ and is only defined on $D_{i}$. Furthermore, $u_{h}^{n,i} \circ \phi_{i}^{-1}$ and $u_{h}^{n,j} \circ \phi_{j}^{-1}$ usually disagree on the overlapping $M_{i} \cap M_{j}$. However, this disagreement is of no importance at all from the viewpoint of approximation. As far as $u_{h}^{n,i}$ approximates $u \circ \phi_{i}$ well on $D_{i}$ for each $i$, we know $u_{h}^{n,i} \circ \phi_{i}^{-1}$ approximates $u$ well on $M_{i}$. Since $M$ is covered by these $M_{i}$, good numerical data would be obtained everywhere on $M$.

\begin{remark}
Algorithm \ref{alg_discrete} implicitly defines a discretization. Actually, we ``discretize" the iteration in Algorithm \ref{alg_continuous} rather than the global problem \eqref{eqn_problem_weak}. On the other hand, this makes a rigorous theoretical analysis more difficult.
\end{remark}

\section{Product Manifolds}\label{sec_product}
Suppose $M$ and $M'$ are compact manifolds with dimensions $d$ and $d'$ respectively. The Cartesian product $M \times M'$ is a compact manifold with dimension $d+d'$. A decomposition of $M$ and another one of $M'$ canonically result in a decomposition of $M \times M'$.
Actually, suppose $M = \bigcup_{i=1}^{m} \mathrm{Int} M_{i}$ and $M' = \bigcup_{i'=1}^{m'} \mathrm{Int} M'_{i'}$, where $M_{i}$ (resp. $M'_{i'}$) are subdomains of $M$ (resp. $M'$). Then
\[
M \times M' = \bigcup_{i=1}^{m} \bigcup_{i'=1}^{m'} \mathrm{Int} (M_{i} \times M'_{i'}),
\]
where $M \times M'$ is decomposed into $mm'$ subdomains $M_{i} \times M'_{i'}$ for $1 \leq i \leq m$ and $1 \leq i' \leq m'$.

This canonical decomposition of $M \times M'$ reflects another advantage of the spaces of rectangular finite elements. More precisely, suppose there are diffeomorphisms $\phi_{i}: D_{i} \rightarrow M_{i}$ and $\phi'_{i'}: D'_{i'} \rightarrow M'_{i'}$, where each $D_{i} \subset \mathbb{R}^{d}$ (resp. $D'_{i'} \subset \mathbb{R}^{d'}$) is a $d$-rectangle (resp. $d'$-rectangle). Then we have the diffeomorphisms
\[
\phi_{i} \times \phi'_{i'}: \ D_{i} \times D'_{i'} \rightarrow M_{i} \times M'_{i'},
\]
where each $D_{i} \times D'_{i'} \subset \mathbb{R}^{d+d'}$ is a $(d+d')$-rectangle. The transition of coordinates between $D_{i} \times D'_{i'}$ and $D_{j} \times D'_{j'}$ is
\[
(\phi_{j} \times \phi'_{j'})^{-1} \circ (\phi_{i} \times \phi'_{i'}) = (\phi_{j}^{-1} \circ \phi_{i}) \times ({\phi'}_{j'}^{-1} \circ \phi'_{i'}).
\]
If rectangular grids are created over both $D_{i}$ and $D'_{i'}$, a rectangular grid over $D_{i} \times D'_{i'}$ follows automatically.

In summary, the procedures of the decomposition and discretization of factor manifolds are helpful for those of product manifolds.

\section{Numerical Experiments}\label{sec_experiment}
We perform several numerical tests of the proposed method on manifolds $S^{4}$, $\mathbb{CP}^{2}$ and $S^{2} \times S^{2}$. They are $4$-dimensional compact manifolds without boundary. While the proposed method applies to problems in all dimensions, the sizes of the linear systems derived from subdomains will increase exponentially with respect to the dimension. On one hand, we would have trouble in the storage of data. On the other hand, we would struggle to find solutions to these linear systems with desired accuracy. This difficulty is so called ``the curse of dimensionality''. It is actually a typical phenomenon of Euclidean spaces rather than of general manifolds. For the sake of presentation and due to the constraint of computing resources, the numerical examples in this paper consider manifolds with dimension no more than $4$. The numerical challenge in higher dimensions will be tackled in a forthcoming future work.

\subsection{Two Problems on \texorpdfstring{$S^{4}$}{S4}}
Let $M= S^{4}$ be the unit sphere in $\mathbb{R}^{5}$, i.e.
\[
M = S^{4} = \left\{ (y_{1}, y_{2}, y_{3}, y_{4}, y_{5}) \in \mathbb{R}^{5} \middle| \sum_{i=1}^{5} y_{i}^{2}  = 1  \right\}.
\]
We decompose $S^{4}$ into two subdomains as follows. By stereographic projections from the south pole $(0,0,0,0,-1)$ and north pole $(0,0,0,0,1)$, we obtain two subdomains $M_{1}$ and $M_{2}$ with coordinates whose interiors cover $S^{4}$. For an illustration please refer to Fig. \ref{fig:stereographic}, where the vertical direction stands for the direction of the 5th coordinate axis, the rectangle $[-r, r]^{4}$ is a domain in $\mathbb{R}^{4} \cong \mathbb{R}^{4} \times \{ 0 \} \subset \mathbb{R}^{5}$, and the intersection of $\mathbb{R}^{4} \times \{ 0 \}$ and $S^{4}$ is the equator of $S^{4}$. For each point $P= (x_{1}, \dots, x_{4}) \in [-r, r]^{4}$, the line segment between $P$ and the north pole $(0,0,0,0,1)$ intersects $S^{4}$ at a single point $Q = (y_{1}, \dots, y_{5})$ other than $(0,0,0,0,1)$. The map $P \mapsto Q$ provides an embedding $\phi_{1}: [-r, r]^{4} \rightarrow S^{4}$. We obtain another embedding $\phi_{2}$ if the north pole is replaced with the south pole. More precisely, we have
\begin{equation}
\label{eq:s4-r}
  D_{1} = D_{2} = [-r, r]^{4} \subset \mathbb{R}^{4}.
\end{equation}
\begin{figure}[htbp]
  \begin{center}
  \includegraphics[width=0.6\textwidth]{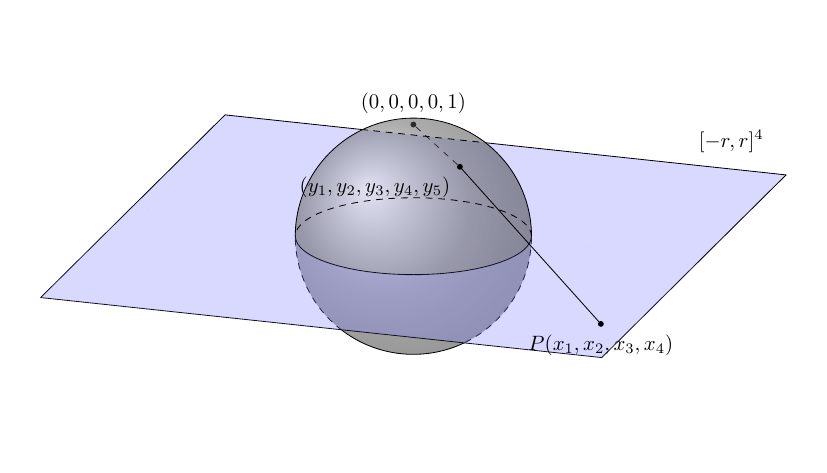}
    \includegraphics[width=0.6\textwidth]{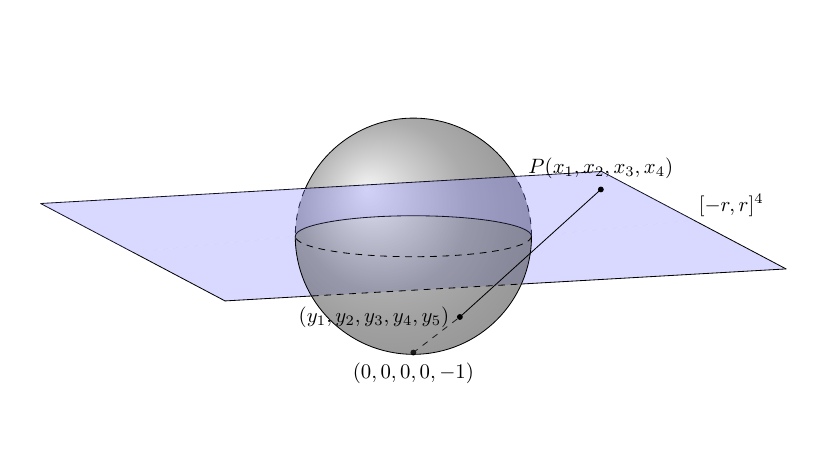}
  \end{center}
\caption{An illustration of the stereographic projections used in the case of $S^4$.}
\label{fig:stereographic}
\end{figure}
Let $x=(x_{1}, x_{2}, x_{3}, x_{4})$ denote the coordinates of $\mathbb{R}^{4}$. Let $\| x \| = \sqrt{\sum_{i=1}^{4} x_{i}^{2}}$. The stereographic projections provide diffeomorphisms $\phi_{i}: D_{i} \rightarrow M_{i}$ as
\begin{eqnarray}
\phi_{1}: \ D_{1} & \rightarrow & M_{1} \subset S^{4} \subset \mathbb{R}^{5} \nonumber
\\
\label{eq:stereographic}
x & \mapsto & \left( \frac{2x}{1 + \|x\|^{2}}, \frac{1 - \|x\|^{2}}{1 + \|x\|^{2}} \right)
\end{eqnarray}
and
\begin{eqnarray*}
\phi_{2}: \ D_{2} & \rightarrow & M_{2} \subset S^{4} \subset \mathbb{R}^{5} \\
x & \mapsto & \left( \frac{2x}{1 + \|x\|^{2}}, \frac{-1 + \|x\|^{2}}{1 + \|x\|^{2}} \right).
\end{eqnarray*}
To guarantee $S^{4} = \bigcup_{i=1}^{2} \mathrm{Int} M_{i}$, we have to let $r>1$. The larger $r$ is, the more overlapping there will be. The transitions of coordinates are given by
\[
\phi_{2}^{-1} \circ \phi_{1} (x) = \phi_{1}^{-1} \circ \phi_{2} (x) = \frac{x}{\| x \|^{2}}.
\]

Equip $S^{4}$ with the Riemannian metric $g$ inherited from the standard one on $\mathbb{R}^{5}$. On each $D_{i}$,
\[
g = 4 (1+ \| x \|^{2})^{-2} \sum_{\alpha = 1}^{4} \mathrm{d} x_{\alpha} \otimes \mathrm{d} x_{\alpha},
\]
and
\begin{eqnarray*}
\Delta v & = & 4^{-1} (1+ \| x \|^{2})^{4} \sum_{\alpha=1}^{4} \frac{\partial}{\partial x_{\alpha}} \left( (1+ \| x \|^{2})^{-2} \frac{\partial v}{\partial x_{\alpha}} \right) \\
& = & 4^{-1} (1+ \| x \|^{2})^{2} \sum_{\alpha=1}^{4} \frac{\partial^{2} v}{\partial x_{\alpha}^{2}} - (1+ \| x \|^{2}) \sum_{\alpha=1}^{4} x_{\alpha} \frac{\partial v}{\partial x_{\alpha}}.
\end{eqnarray*}

Consider the model problem \eqref{eqn_problem} on $S^{4}$ with $b>0$. Choose the true solution to \eqref{eqn_problem} as
\[
u= y_{5},
\]
where $y_{5}$ is the $5$th coordinate of $\mathbb{R}^{5}$. Then
\[
f=(4+b) u
\]
in \eqref{eqn_problem}. On $D_{i}$, $u$ has the expression
\[
u \circ \phi_{1} (x) = \frac{1 - \|x\|^{2}}{1 + \|x\|^{2}}, \qquad u \circ \phi_{2} (x) = \frac{-1 + \|x\|^{2}}{1 + \|x\|^{2}}.
\]
The weak form of \eqref{eqn_problem} on $D_{i}$ is formulated as: $\forall v \in H^{1}_{0} (D_{i})$,
\begin{eqnarray*}
&  & \int_{D_{i}} 4 (1+ \| x \|^{2})^{-2} \sum_{\alpha =1}^{4} \frac{\partial u \circ \phi_{i}}{\partial x_{\alpha}} \frac{\partial v}{\partial x_{\alpha}} \mathrm{d} x_{1} \mathrm{d} x_{2} \mathrm{d} x_{3} \mathrm{d} x_{4} \\
&  &  + \int_{D_{i}} 16 (1+ \| x \|^{2})^{-4} b \cdot u \circ \phi_{i} \cdot v \mathrm{d} x_{1} \mathrm{d} x_{2} \mathrm{d} x_{3} \mathrm{d} x_{4} \\
& = & \int_{D_{i}} 16 (1+ \| x \|^{2})^{-4} \cdot f \circ \phi_{i} \cdot v \mathrm{d} x_{1} \mathrm{d} x_{2} \mathrm{d} x_{3} \mathrm{d} x_{4}.
\end{eqnarray*}

Now we choose $b$ in \eqref{eqn_problem} as $1$. For the discretization, we divide each coordinate interval $[-r,r]$ into $N$ equal parts. The scale of the grid is thus $h=2r/N$. There are $(N+1)^4$ nodes on $D_{i}$, most rows of the stiffness matrix have $3^{4} =81$ nonzero entries. We keep $N \leq 80$ due to the memory limitation of the hardware.

To get the $n$-th discrete approximation $u_{h}^{n} = (u_{h}^{n,1}, u_{h}^{n,2})$, we need to solve a linear system $A_{i} X^{n,i} =b_{n,i}$ for $i=1,2$, where $X^{n,i}$ provides the interior degrees of freedom of $u_{h}^{n,i}$. We use the Conjugate Gradient Method (CG) to find $X^{n,i}$. As a result, the process to generate the sequence $\{ u_{h}^{n} \}$ is a nested iteration. The outer iteration is the DDM procedure Algorithm \ref{alg_discrete}. The initial guess is chosen as $u_{h}^{0} =0$. For each $n$, the inner iteration is the CG iteration to solve $A_{i} X^{n,i} =b_{n,i}$ for $i=1,2$. Note that $A_{i}$ remains the same when $n$ changes, whereas $b_{n,i}$ varies because of the evaluation of $u_{h}^{n,i}|_{\partial D_{i}}$. If $\{ u_{h}^{n} \}$ does converge, $X^{n-1,i}$ will be close to $X^{n,i}$ when $n$ is large enough. Thus, we choose the initial guess of $X^{n,i}$ as $X^{n-1,i}$. The tolerance for CG is set as
\[
\| A_{i} X^{n,i} - b_{n,i} \|_{2} / \| b_{n,i} \|_{2} \leq 10^{-8}.
\]
Our numerical results show that $u_{h}^{n}$ becomes stable when $n=n_{0}$ for some $n_{0}$, i.e. $u_{h}^{n} = u_{h}^{n_{0}}$ up to machine precision for all $n \geq n_{0}$. Actually, if
\[
\| A_{i} X^{n,i} - b_{n+1,i} \|_{2} / \| b_{n+1,i} \|_{2} \leq 10^{-8}
\]
for all $i$, then the inner iteration terminates for $n+1$ and $X^{n+1,i} = X^{n,i}$. We found that inner iteration terminates for all $n > n_{0}$. In other words, practically, the sequence $\{ u_{h}^{n} \}$ reaches its limit
\[
u_{h}^{\infty} = u_{h}^{n_{0}}
\]
at step $n_{0}$.

In the following tables,
\[
I_{h} u = (I_{h} u^{1}, I_{h} u^{2}) \in V_{h},
\]
where $I_{h} u^{i} \in V_{h}^{i}$ is the interpolation of $u \circ \phi_{i}$ on $D_{i}$. We define the energy norm of the error as
\[
\| I_{h} u - u_{h}^{\infty} \|_{a} = \max \{ a_{i}(I_{h} u^{i} - u_{h}^{\infty,i}, I_{h} u^{i} - u_{h}^{\infty,i})^{\frac{1}{2}} \mid i=1,2 \}.
\]
The $L^{2}$-norm $\| I_{h} u - u_{h}^{\infty} \|_{L^{2}}$, $L^{\infty}$-norm $\| I_{h} u - u_{h}^{\infty} \|_{L^{\infty}}$ and $H^{1}$-seminorm $| I_{h} u - u_{h}^{\infty} |_{H^{1}}$ are defined in similar ways. The numerical results are as follows in Tables \ref{table:s4-r1.2} and \ref{table:s4-r2}, where, for each norm, the data on the left side of each cell are errors and orders of convergence are appended to the right.

\begin{table}[ht]
\begin{center}
\begin{tabular}{|c|cc|cc|cc|cc|c|}
\hline
$h$ & \multicolumn{2}{c|}{$\|I_{h} u -  u_{h}^{\infty} \|_{L^{\infty}}$} & \multicolumn{2}{c|}{$\|I_{h} u -  u_{h}^{\infty} \|_{L^{2}}$} & \multicolumn{2}{c|}{$| I_{h} u - u_{h}^{\infty} |_{H^{1}}$} & \multicolumn{2}{c|}{$\| I_{h} u - u_{h}^{\infty} \|_{a}$} & $n_{0}$ \\
\hline
$0.24$ & $0.0302$ &  & $0.0690$ & & $0.2348$ & & $0.1830$ & & $22$ \\
\hline
$0.12$ & $0.0095$ & $1.7$ & $0.0180$ & $1.9$ & $0.0717$ & $1.7$ & $0.0501$ & $1.9$ & $23$ \\
\hline
$0.06$ & $0.0032$ & $1.6$ & $0.0046$ & $2.0$ & $0.0239$ & $1.6$ & $0.0150$ & $1.7$ & $22$ \\
\hline
$0.03$ & $7.2393e-4$ & $2.2$ & $0.0011$ & $2.0$ & $0.0082$ & $1.5$ & $0.0048$ & $1.6$ & $22$ \\
\hline
\end{tabular}
\end{center}
\caption{Convergence result on $S^4$ for $(u,r)= (y_{5}, 1.2)$.}
\label{table:s4-r1.2}
\end{table}

\begin{table}[ht]
\begin{center}
\begin{tabular}{|c|cc|cc|cc|cc|c|}
\hline
 $h$ & \multicolumn{2}{c|}{$\|I_{h} u -  u_{h}^{\infty} \|_{L^{\infty}}$} & \multicolumn{2}{c|}{$\|I_{h} u -  u_{h}^{\infty} \|_{L^{2}}$} & \multicolumn{2}{c|}{$| I_{h} u - u_{h}^{\infty} |_{H^{1}}$} & \multicolumn{2}{c|}{$\| I_{h} u - u_{h}^{\infty} \|_{a}$} & $n_{0}$ \\
\hline
$0.4$ & $0.1459$ & & $1.2578$ & & $0.9782$ & & $0.5725$ & & $10$ \\
\hline
$0.2$ & $0.0458$ & $1.7$ & $0.2546$ & $2.3$ & $0.2927$ & $1.7$ & $0.1416$ & $2.0$ & $10$ \\
\hline
$0.1$ & $0.0110$ & $2.1$ & $0.0665$ & $1.9$ & $0.1199$ & $1.3$ & $0.0427$ & $1.7$ & $10$ \\
\hline
$0.05$ & $0.0031$ & $1.8$ & $0.0165$ & $2.0$ & $0.0432$ & $1.5$ & $0.0131$ & $1.7$ & $10$ \\
\hline
\end{tabular}
\end{center}
\caption{Convergence result on $S^4$  for $(u,r)= (y_{5}, 2)$.}
\label{table:s4-r2}
\end{table}

We see that the error $I_{h} u -  u_{h}^{\infty}$ decays in the optimal order when $h$ decreases. Furthermore, $n_{0}$ decreases when $r$ becomes larger, i.e. $u_{h}^{n}$ reaches its limit $u_{h}^{\infty}$ fast provides that the overlapping between subdomains is large.

\begin{remark}
As shown in Tables \ref{table:s4-r1.2}, \ref{table:s4-r2}, and other tables below, the convergences under $H^{1}$- and the energy norms are significantly better than the optimal first order in $h$. The reason is yet to be explored. On the other hand, since the true solution to our example is $C^{\infty}$, the finite element spaces are defined on highly symmetric grids (rectangular), the transition maps are smooth, thus any (or all) from these factors may contribute to superconvergence. But we cannot prove this hypothesis at this stage.
\end{remark}

Note that Tables \ref{table:s4-r1.2} and \ref{table:s4-r2} actually show $u_{h}^{\infty,i} \circ \phi_{i}^{-1}$ approximates $u|_{M_{i}}$ well. Hence we obtain good numerical approximation to $u$ everywhere on $M$ because $M = \bigcup_{i=1}^{m} M_{i}$. More precisely, there is a $1-1$ and onto correspondence between functions on $M_{i}$ and those on $D_{i}$, i.e. a function $v$ on $M_{i}$ correspondence to $v \circ \phi_{i}$ on $D_{i}$. Via this bijection, $L^{\infty} (M_{i})$ (resp. $L^{2} (M_{i})$ and $H^{1} (M_{i})$) is \textit{isometrically isomorphic} to $L^{\infty} (D_{i})$ (resp. $L^{2} (D_{i}; g)$ and $H^{1} (D_{i}; g)$). Here the notation ``$g$" stands for the metric tensor in \eqref{eqn_metric}, $L^{2} (D_{i}; g)$ and $H^{1} (D_{i}; g)$ are the $L^{2}$-space and $H^{1}$-space respectively on $D_{i}$ with norms
\[
\| w \|_{L^{2} (D_{i}; g)} = \left( \int_{D_{i}} |w|^{2} \sqrt{G} \mathrm{d} x_{1} \cdots \mathrm{d} x_{d} \right)^{\frac{1}{2}},
\]
\[
|w|_{H^{1} (D_{i}; g)} = \left( \int_{D_{i}} \sum_{\alpha,\beta=1}^{d} g^{\alpha \beta} \frac{\partial w}{\partial x_{\alpha}} \frac{\partial w}{\partial x_{\beta}} \sqrt{G} \mathrm{d} x_{1} \cdots \mathrm{d} x_{d} \right)^{\frac{1}{2}},
\]
and
\[
\| w \|_{H^{1} (D_{i}; g)} = \left( \| w \|_{L^{2} (D_{i}; g)}^{2} + |w|_{H^{1} (D_{i}; g)}^{2} \right)^{\frac{1}{2}}.
\]
The above ``isometrically isomorphic" means the bijection is a linear isomorphism preserving norms (see Definition 1.13 in \cite[p.~66]{conway}), i.e. $\| v \|_{L^{2} (M_{i})} = \| v \circ \phi_{i} \|_{L^{2} (D_{i}; g)}$ and so on. Meanwhile, by \eqref{eqn_energy}, we also have $a(v,v)_{M_{i}} = a(v \circ \phi_{i}, v \circ \phi_{i})_{D_{i}}$. Furthermore, $(g_{\alpha \beta})_{d \times d}$ is bounded and uniformly elliptic on $D_{i}$ because $(g_{\alpha \beta})_{d \times d}$ is $C^{\infty}$ and $D_{i}$ is compact. Thus the $\| \cdot \|_{L^{2} (D_{i}; g)}$ and $|\cdot|_{H^{1} (D_{i}; g)}$ are equivalent to the usual $\| \cdot \|_{L^{2} (D_{i})}$ and $|\cdot|_{H^{1} (D_{i})}$ respectively. Therefore, $u_{h}^{\infty,i} \circ \phi_{i}^{-1}$ approximates $I_{h} u^{i} \circ \phi_{i}^{-1}$ well in $\| \cdot \|_{L^{\infty} (M_{i})}$, $\| \cdot \|_{L^{2} (M_{i})}$, $|\cdot|_{H^{1} (M_{i})}$ and $a(\cdot, \cdot)_{M_{i}}$ as far as $u_{h}^{\infty,i}$ approximates $I_{h} u^{i}$ well in $\| \cdot \|_{L^{\infty} (D_{i})}$, $\| \cdot \|_{L^{2} (D_{i})}$, $|\cdot|_{H^{1} (D_{i})}$ and $a(\cdot, \cdot)_{D_{i}}$. Since $I_{h} u^{i} \circ \phi_{i}^{-1}$ is an interpolation of $u|_{M_{i}}$, we infer $u_{h}^{\infty,i} \circ \phi_{i}^{-1}$ is a good approximation of $u|_{M_{i}}$.

We also investigated the number of iterations required to achieve an approximation of the same order of accuracy as $I_{h} u -  u_{h}^{\infty}$. So we set a tolerance for the outer iteration as
\[
\|I_{h} u -  u_{h}^{n} \|_{L^{\infty}} \leq 2 \|I_{h} u -  u_{h}^{\infty} \|_{L^{\infty}}.
\]
The numerical results are as follows in Tables \ref{table:s4-r1.2-tl} and \ref{table:s4-r2-tl}. We see that $n$ are much less than $n_{0}$.
\begin{table}[ht]
\begin{center}
\begin{tabular}{|c|c|c|c|c|c|}
\hline
$h$ & $\|I_{h} u -  u_{h}^{n} \|_{L^{\infty}}$ & $\|I_{h} u -  u_{h}^{n} \|_{L^{2}}$ & $| I_{h} u - u_{h}^{n} |_{H^{1}}$ & $\| I_{h} u - u_{h}^{n} \|_{a}$ & $n$ \\
\hline
$0.24$ & $0.0569$ & $0.2066$ & $0.2604$ & $0.2193$ & $4$ \\
\hline
$0.12$ & $0.0142$ & $0.0436$ & $ 0.0756$ & $0.0554$ & $6$ \\
\hline
$0.06$ & $0.0052$ & $0.0158$ & $0.0253$ & $0.0179$ & $7$ \\
\hline
$0.03$ & $0.0011$ & $0.0033$ & $0.0084$ & $0.0051$ & $9$ \\
\hline
\end{tabular}
\end{center}
\caption{Convergence result on $S^4$ for $(u,r)= (y_{5}, 1.2)$.}
\label{table:s4-r1.2-tl}
\end{table}

\begin{table}[ht]
\begin{center}
\begin{tabular}{|c|c|c|c|c|c|}
\hline
$h$ & $\|I_{h} u -  u_{h}^{n} \|_{L^{\infty}}$ & $\|I_{h} u -  u_{h}^{n} \|_{L^{2}}$ & $| I_{h} u - u_{h}^{n} |_{H^{1}}$ & $\| I_{h} u - u_{h}^{n} \|_{a}$ & $n$ \\
\hline
$0.4$ & $0.2231$ & $2.3141$ & $1.0592$ & $0.7099$ & $2$ \\
\hline
$0.2$ & $0.0550$ & $0.3806$ & $0.2953$ & $ 0.1551$ & $3$ \\
\hline
$0.1$ & $0.0203$ & $0.1945$ & $0.1281$ & $0.0608$ & $3$ \\
\hline
$0.05$ & $0.0042$ & $0.0315$ & $0.0434$ & $0.0144$ & $4$ \\
\hline
\end{tabular}
\end{center}
\caption{Convergence result on $S^4$ for $(u,r)= (y_{5}, 2)$.}
\label{table:s4-r2-tl}
\end{table}

Now we consider a second problem on $S^{4}$ with true solution $u= y_{1} y_{5}$ in \eqref{eqn_problem}. Then $f= (10+b) u$. On $D_{i}$, $u$ has the expression
\[
u \circ \phi_{1} (x) = \frac{2x_{1} (1 - \|x\|^{2})}{(1 + \|x\|^{2})^{2}}, \qquad u \circ \phi_{2} (x) = \frac{2x_{1}(-1 + \|x\|^{2})}{(1 + \|x\|^{2})^{2}}.
\]
We choose the $b$ in \eqref{eqn_problem} as $1$. The numerical results are in tables \ref{table:s4-qd-r1.2}, \ref{table:s4-qd-r2}, \ref{table:s4-qd-r1.2-tl} and \ref{table:s4-qd-r2-tl}. The performance of our algorithm on this problem is similar to that of the first one.

\begin{table}[ht]
\begin{center}
\begin{tabular}{|c|cc|cc|cc|cc|c|}
\hline
$h$ & \multicolumn{2}{c|}{$\|I_{h} u -  u_{h}^{\infty} \|_{L^{\infty}}$} & \multicolumn{2}{c|}{$\|I_{h} u -  u_{h}^{\infty} \|_{L^{2}}$} & \multicolumn{2}{c|}{$| I_{h} u - u_{h}^{\infty} |_{H^{1}}$} & \multicolumn{2}{c|}{$\| I_{h} u - u_{h}^{\infty} \|_{a}$} & $n_{0}$ \\
\hline
$0.24$ & $0.0445$ &  & $0.0782$ & & $0.2142$ & & $0.1633$ & & $9$ \\
\hline
$0.12$ & $0.0121$ & $1.9$  & $0.0200$ & $2.0$ & $0.0666$ & $1.7$ & $0.0450$ & $1.9$ & $9$ \\
\hline
$0.06$ & $0.0031$ & $2.0$ & $0.0051$ & $2.0$ & $0.0223$ & $1.6$ & $0.0136$ & $1.7$ & $9$ \\
\hline
$0.03$ & $7.8553e-4$ & $2.0$ & $0.0013$ & $2.0$ & $0.0077$ & $1.5$ & $0.0043$ & $1.7$ & $9$ \\
\hline
\end{tabular}
\end{center}
\caption{Convergence result on $S^4$ for $(u,r)= (y_{1} y_{5}, 1.2)$.}
\label{table:s4-qd-r1.2}
\end{table}

\begin{table}[ht]
\begin{center}
\begin{tabular}{|c|cc|cc|cc|cc|c|}
\hline
$h$ & \multicolumn{2}{c|}{$\|I_{h} u -  u_{h}^{\infty} \|_{L^{\infty}}$} & \multicolumn{2}{c|}{$\|I_{h} u -  u_{h}^{\infty} \|_{L^{2}}$} & \multicolumn{2}{c|}{$| I_{h} u - u_{h}^{\infty} |_{H^{1}}$} & \multicolumn{2}{c|}{$\| I_{h} u - u_{h}^{\infty} \|_{a}$} & $n_{0}$ \\
\hline
$0.4$ & $0.1389$ &  & $1.0971$ & & $1.1316$ & & $0.5017$ & & $4$ \\
\hline
$0.2$ & $0.0478$ & $1.5$ & $0.2658$ & $2.0$ & $0.3540$ & $1.7$ & $0.1423$ & $1.8$ & $4$ \\
\hline
$0.1$ & $0.0135$ & $1.8$ & $0.0701$ & $1.9$ & $0.1141$ & $1.6$ & $0.0401$ & $1.8$ & $5$ \\
\hline
$0.05$ & $0.0034$ & $2.0$ & $0.0176$ & $2.0$ & $0.0375$ & $1.6$ & $0.0117$ & $1.8$ & $5$ \\
\hline
\end{tabular}
\end{center}
\caption{Convergence result on $S^4$ for $(u,r)= (y_{1} y_{5}, 2)$.}
\label{table:s4-qd-r2}
\end{table}

\begin{table}[ht]
\begin{center}
\begin{tabular}{|c|c|c|c|c|c|}
\hline
$h$ & $\|I_{h} u -  u_{h}^{n} \|_{L^{\infty}}$ & $\|I_{h} u -  u_{h}^{n} \|_{L^{2}}$ & $| I_{h} u - u_{h}^{n} |_{H^{1}}$ & $\| I_{h} u - u_{h}^{n} \|_{a}$ & $n$ \\
\hline
$0.24$ & $0.0551$ & $0.1201$ & $0.2379$ & $0.1933$ & $2$ \\
\hline
$0.12$ & $0.0128$ & $0.0235$ & $0.0676$ & $0.0468$ & $3$ \\
\hline
$0.06$ & $0.0040$ & $0.0087$ & $0.0240$ & $0.0161$ & $3$ \\
\hline
$0.03$ & $8.3731e-4$ & $0.0016$ & $0.0077$ & $0.0044$ & $4$ \\
\hline
\end{tabular}
\end{center}
\caption{Convergence result on $S^4$ for $(u,r)= (y_{1} y_{5}, 1.2)$.}
\label{table:s4-qd-r1.2-tl}
\end{table}

\begin{table}[ht]
\begin{center}
\begin{tabular}{|c|c|c|c|c|c|}
\hline
$h$ & $\|I_{h} u -  u_{h}^{n} \|_{L^{\infty}}$ & $\|I_{h} u -  u_{h}^{n} \|_{L^{2}}$ & $| I_{h} u - u_{h}^{n} |_{H^{1}}$ & $\| I_{h} u - u_{h}^{n} \|_{a}$ & $n$ \\
\hline
$0.4$ & $0.1393$ & $1.1006$ & $1.1349$ & $0.5028$ & $2$ \\
\hline
$0.2$ & $0.0484$ & $0.2701$ & $0.3580$ & $0.1438$ & $2$ \\
\hline
$0.1$ & $0.0142$ & $0.0747$ & $0.1174$ & $0.0416$ & $2$ \\
\hline
$0.05$ & $0.0043$ & $0.0222$ & $0.0403$ & $0.0131$ & $2$ \\
\hline
\end{tabular}
\end{center}
\caption{Convergence result on $S^4$ for $(u,r)= (y_{1} y_{5}, 2)$.}
\label{table:s4-qd-r2-tl}
\end{table}

\subsection{A Problem on \texorpdfstring{$\mathbb{CP}^{2}$}{CP2}}
Let $M = \mathbb{CP}^{2}$ be the complex projective plane. It is a compact complex manifold with complex dimension $2$. Certainly, it can be considered as a real manifold with dimension $4$.

Unlike $S^{4}$, the $\mathbb{CP}^{2}$ is not a submanifold of any Euclidean space by definition. Furthermore, $\mathbb{CP}^{2}$ cannot be embedded differential-topologically into $\mathbb{R}^{k}$ with $k<7$ by the theory of characteristic classes (\cite[Corollary~11.4]{milnor_stasheff}). Whitney constructed an explicit embedding of $\mathbb{CP}^{2}$ into $\mathbb{R}^{7}$ in an ingenious way (\cite[Appendix]{whitney44}). Since the codimension of $\mathbb{CP}^{2}$ in $\mathbb{R}^{7}$ is $3$, it is incredibly difficult to build effective polytopal approximations to $\mathbb{CP}^{2}$ in $\mathbb{R}^{7}$. On the other hand, by definition, $\mathbb{CP}^{2}$ can be constructed by patching together three coordinate charts, where the transitions of coordinates have explicit and neat formulas. Thus, it is very suitable to apply our method to $\mathbb{CP}^{2}$.

The $\mathbb{CP}^{2}$ can be easily defined as a quotient space. Let
\[
\mathbb{C}^{3} \setminus \{ \mathbf{0} \} = \{ (w_{0}, w_{1}, w_{2}) \mid \mathbf{0} \neq (w_{0}, w_{1}, w_{2}) \in \mathbb{C}^{3} \}.
\]
Here $\mathbf{0} \in \mathbb{C}^{3}$ is the origin, each $w_{j}$ is a complex number for $0 \leq j \leq 2$, and, following the convention of algebraic geometry, the index $j$ starts from $0$ rather than $1$. Define a relation of equivalence on $\mathbb{C}^{3} \setminus \{ \mathbf{0} \}$ as
\[
(w_{0}, w_{1}, w_{2}) \sim (w'_{0}, w'_{1}, w'_{2})
\]
if and only if
\[
(w_{0}, w_{1}, w_{2}) = \lambda (w'_{0}, w'_{1}, w'_{2})
\]
for some $0 \neq \lambda \in \mathbb{C}$. Define
\[
\mathbb{CP}^{2} = \mathbb{C}^{3} \setminus \{ \mathbf{0} \} / \sim.
\]
Thus, every $P \in \mathbb{CP}^{2}$ can be represented by a vector $(w_{0}, w_{1}, w_{2}) \in \mathbb{C}^{3} \setminus \{ \mathbf{0} \}$. Conventionally, we write
\[
P = [w_{0}, w_{1}, w_{2}],
\]
where $[w_{0}, w_{1}, w_{2}]$ are called the \textit{homogeneous coordinates} of $P$. Note that, for $\lambda \neq 0$,
\[
[w_{0}, w_{1}, w_{2}] = [\lambda w_{0}, \lambda w_{1}, \lambda w_{2}].
\]
For more details of general $\mathbb{CP}^{k}$, see \cite[p.~15]{griffiths_harris}.

Now we decompose $\mathbb{CP}^{2}$ into three subdomains $M_{j}$ for $0 \leq j \leq 2$ (note that the index $j$ is chosen to start from $0$ for the convenience of presentation.) In the following, $z_{j} = x_{j} + \sqrt{-1} y_{j} \in \mathbb{C}$, $x_{j} \in \mathbb{R}$, and $y_{j} \in \mathbb{R}$. We shall identify the complex number $z_{j}$ with the $2$-dimensional real vector $(x_{j}, y_{j})$. Let $D_{j} = [-r, r]^{4} \subset \mathbb{R}^{4} \simeq \mathbb{C}^{2}$. For $0 \leq j \leq 2$, we have the following diffeomorphisms
\begin{eqnarray*}
\phi_{0}: \ D_{0} & \rightarrow & M_{0} \subset  \mathbb{CP}^{2}  \\
(z_{1}, z_{2}) & \mapsto & [1, z_{1}, z_{2}],
\end{eqnarray*}
\begin{eqnarray*}
\phi_{1}: \ D_{1} & \rightarrow & M_{1} \subset  \mathbb{CP}^{2} \\
(z_{0}, z_{2}) & \mapsto & [z_{0}, 1, z_{2}],
\end{eqnarray*}
and
\begin{eqnarray*}
\phi_{2}: \ D_{2} & \rightarrow & M_{2} \subset  \mathbb{CP}^{2} \\
(z_{0}, z_{1}) & \mapsto & [z_{0}, z_{1}, 1].
\end{eqnarray*}
To guarantee $\mathbb{CP}^{2} = \bigcup_{j=0}^{2} \mathrm{Int} M_{j}$, we have to let $r>1$. The larger $r$ is, the more overlapping there will be. The transitions of coordinates are given by
\[
\phi_{1}^{-1} \circ \phi_{0} (z_{1}, z_{2}) = \left( \frac{1}{z_{1}}, \frac{z_{2}}{z_{1}} \right), \qquad \phi_{0}^{-1} \circ \phi_{1} (z_{0}, z_{2}) = \left( \frac{1}{z_{0}}, \frac{z_{2}}{z_{0}} \right),
\]
\[
\phi_{2}^{-1} \circ \phi_{0} (z_{1}, z_{2}) = \left( \frac{1}{z_{2}}, \frac{z_{1}}{z_{2}} \right), \qquad \phi_{0}^{-1} \circ \phi_{2} (z_{0}, z_{1}) = \left( \frac{z_{1}}{z_{0}}, \frac{1}{z_{0}} \right),
\]
\[
\phi_{2}^{-1} \circ \phi_{1} (z_{0}, z_{2}) = \left( \frac{z_{0}}{z_{2}}, \frac{1}{z_{2}} \right), \qquad \phi_{1}^{-1} \circ \phi_{2} (z_{0}, z_{1}) = \left( \frac{z_{0}}{z_{1}}, \frac{1}{z_{1}} \right).
\]

Equipping it with the classical \textit{Fubini-Study metric} (c.f. \cite[p.~30]{griffiths_harris}), $\mathbb{CP}^{2}$ becomes a K\"{a}hler manifold with K\"{a}hler form
\[
\frac{\sqrt{-1}}{2} \partial \bar{\partial} \log \sum_{j=0}^{2} |w_{j}|^{2},
\]
where $[w_{0}, w_{1}, w_{2}]$ are the homogeneous coordinates of $\mathbb{CP}^{2}$. The \textit{Fubini-Study metric}, denoted by $\mathcal{H}$, is a Hermitian metric. On each $D_{j}$, it is expressed as
\[
\mathcal{H} = (1+ \|z\|^{2})^{-1} \sum_{\alpha = 0, \alpha \neq j}^{2} \mathrm{d} z_{\alpha} \otimes \mathrm{d} \bar{z}_{\alpha}
- (1+ \|z\|^{2})^{-2} \sum_{\alpha = 0, \alpha \neq j}^{2} \sum_{\beta = 0, \beta \neq j}^{2} \bar{z}_{\alpha} z_{\beta} \mathrm{d} z_{\alpha} \otimes \mathrm{d} \bar{z}_{\beta},
\]
where
\begin{equation}\label{eqn_projective_norm}
\|z\|^{2} = \sum_{\alpha = 0, \alpha \neq j}^{2} |z_{\alpha}|^{2} = \sum_{\alpha = 0, \alpha \neq j}^{2} (x_{\alpha}^{2} + y_{\alpha}^{2}).
\end{equation}
We choose the Riemannian metric $g$ on $\mathbb{CP}^{2}$ as the real part of $\mathcal{H}$. This $g$ provides the underlying Riemannian structure of the above K\"{a}hler structure. The Laplacian is expressed as
\begin{eqnarray*}
&  & \Delta v = 2 \Delta_{\partial} v = 2 \Delta_{\bar{\partial}} v \\
& = & 4 (1+ \|z\|^{2})^{3} \sum_{\alpha = 0, \alpha \neq j}^{2} \frac{\partial}{\partial z_{\alpha}} \left( (1+ \|z\|^{2})^{-2} \frac{\partial v}{\partial \bar{z}_{\alpha}} + (1+ \|z\|^{2})^{-2} z_{\alpha} \sum_{\beta = 0, \beta \neq j}^{2} \bar{z}_{\beta} \frac{\partial v}{\partial \bar{z}_{\beta}} \right) \\
& = & 4 (1+ \|z\|^{2}) \left( \sum_{\alpha = 0, \alpha \neq j}^{2} \frac{\partial^{2} v}{\partial z_{\alpha} \partial \bar{z}_{\alpha}} + \sum_{\alpha = 0, \alpha \neq j}^{2} \sum_{\beta = 0, \beta \neq j}^{2} z_{\alpha} \bar{z}_{\beta} \frac{\partial^{2} v}{\partial z_{\alpha} \partial \bar{z}_{\beta}} \right).
\end{eqnarray*}

We consider the model problem \eqref{eqn_problem} with $b>0$. Choose constants $a_{j} \in \mathbb{R}$, $0 \leq j \leq 2$. Choose the true solution to \eqref{eqn_problem} as
\begin{equation}\label{eqn_u_cp2}
u([w_{0}, w_{1}, w_{2}]) = \sum_{j=0}^{2} a_{j} |w_{j}|^{2},
\end{equation}
where $[w_{0}, w_{1}, w_{2}]$ are homogeneous coordinates with normalization $\sum_{j=0}^{2} |w_{j}|^{2} =1$. It is easy to see that $u$ is well-defined. The $f$ in \eqref{eqn_problem} is then
\[
f = (12+b) u - 4 \sum_{j=0}^{2} a_{j}.
\]
On $D_{j}$, the true solution $u$ has the expression
\[
u \circ \phi_{j} = \frac{a_{j} + \sum_{\beta = 0, \beta \neq j}^{2} a_{\beta} |z_{\beta}|^{2}}{1 + \sum_{\beta = 0, \beta \neq j}^{2} |z_{\beta}|^{2}} = \frac{a_{j} + \sum_{\beta = 0, \beta \neq j}^{2} a_{\beta} (x_{\beta}^{2} + y_{\beta}^{2})}{1 + \sum_{\beta = 0, \beta \neq j}^{2} (x_{\beta}^{2} + y_{\beta}^{2})}.
\]
The weak form of \eqref{eqn_problem} on $D_{j}$ is formulated as: $\forall v \in H^{1}_{0} (D_{j})$,
\begin{eqnarray*}
&  & \int_{D_{j}} (1 + \|z\|^{2})^{-2} \sum_{\alpha = 0, \alpha \neq j}^{2} \left( \frac{\partial u \circ \phi_{j}}{\partial x_{\alpha}} \frac{\partial v}{\partial x_{\alpha}} + \frac{\partial u \circ \phi_{j}}{\partial y_{\alpha}} \frac{\partial v}{\partial y_{\alpha}} \right) \\
&  & + \int_{D_{j}} (1 + \|z\|^{2})^{-2} \left[ \sum_{\alpha = 0, \alpha \neq j}^{2} \left( x_{\alpha} \frac{\partial u \circ \phi_{j}}{\partial x_{\alpha}} + y_{\alpha} \frac{\partial u \circ \phi_{j}}{\partial y_{\alpha}} \right) \right] \cdot \left[ \sum_{\alpha = 0, \alpha \neq j}^{2} \left( x_{\alpha} \frac{\partial v}{\partial x_{\alpha}} + y_{\alpha} \frac{\partial v}{\partial y_{\alpha}} \right) \right] \\
&  & + \int_{D_{j}} (1 + \|z\|^{2})^{-2} \left[ \sum_{\alpha = 0, \alpha \neq j}^{2} \left( y_{\alpha} \frac{\partial u \circ \phi_{j}}{\partial x_{\alpha}} - x_{\alpha} \frac{\partial u \circ \phi_{j}}{\partial y_{\alpha}} \right) \right] \cdot \left[ \sum_{\alpha = 0, \alpha \neq j}^{2} \left( y_{\alpha} \frac{\partial v}{\partial x_{\alpha}} - x_{\alpha} \frac{\partial v}{\partial y_{\alpha}} \right) \right] \\
&  & + \int_{D_{j}} (1 + \|z\|^{2})^{-3} b \cdot u \circ \phi_{j} \cdot v \\
& = & \int_{D_{j}} (1 + \|z\|^{2})^{-3} f \circ \phi_{j} \cdot v,
\end{eqnarray*}
where $\| z \|^{2}$ is defined in \eqref{eqn_projective_norm}, and the symbols $\mathrm{d} x_{\alpha}$ and $\mathrm{d} y_{\alpha}$ in the integrals are also omitted.

Now we choose $b$ in \eqref{eqn_problem} as $4$, choose $(a_{0}, a_{1}, a_{2})$ in \eqref{eqn_u_cp2} as $(0,1,-1)$. The numerical results are as follows in Tables \ref{table:cp2-r1.2}, \ref{table:cp2-r2}, \ref{table:cp2-r1.2-tl} and \ref{table:cp2-r2-tl}. Tables \ref{table:cp2-r1.2} and \ref{table:cp2-r2} indicate that the sequence $\{ u_{h}^{n} \}$ practiaclly reaches its limit $u_{h}^{\infty}$ at step $n_{0}$. The convergence rate improves as the overlapping between subdomains increases. By referring to Tables \ref{table:cp2-r1.2-tl} and \ref{table:cp2-r2-tl}, it is possible to achieve an approximation with the same order of accuracy as $I_{h} u -  u_{h}^{\infty}$ with much fewer iterative steps.

\begin{table}[ht]
\begin{center}
\begin{tabular}{|c|cc|cc|cc|cc|c|} \hline
$h$ & \multicolumn{2}{c|}{$\|I_{h} u -  u_{h}^{\infty} \|_{L^{\infty}}$} & \multicolumn{2}{c|}{$\|I_{h} u -  u_{h}^{\infty} \|_{L^{2}}$} & \multicolumn{2}{c|}{$| I_{h} u - u_{h}^{\infty} |_{H^{1}}$} & \multicolumn{2}{c|}{$\| I_{h} u - u_{h}^{\infty} \|_{a}$} & $n_{0}$ \\
\hline
$0.24$ & $0.0376$ & & $0.0454$ & & $0.1559$ & & $0.0718$ & & $38$ \\
\hline
$0.12$ & $0.0103$ & $1.9$ & $0.0116$ & $2.0$ & $0.0441$ & $1.8$ & $0.0204$ & $1.8$ & $36$ \\
\hline
$0.06$ & $0.0024$ & $2.1$ & $0.0029$ & $2.0$ & $0.0127$ & $1.8$ & $0.0057$ & $1.8$ & $35$ \\
\hline
$0.03$ & $6.1832e-4$ & $2.0$ & $7.3810e-4$ & $2.0$ & $0.0039$ & $1.7$ & $0.0017$ & $1.7$ & $34$ \\
\hline
\end{tabular}
\end{center}
\caption{Convergence result on $\mathbb{CP}^{2}$ for $r=1.2$.}
\label{table:cp2-r1.2}
\end{table}

\begin{table}[ht]
  \begin{center}
\begin{tabular}{|c|cc|cc|cc|cc|c|} \hline
$h$ & \multicolumn{2}{c|}{$\|I_{h} u -  u_{h}^{\infty} \|_{L^{\infty}}$} & \multicolumn{2}{c|}{$\|I_{h} u -  u_{h}^{\infty} \|_{L^{2}}$} & \multicolumn{2}{c|}{$| I_{h} u - u_{h}^{\infty} |_{H^{1}}$} & \multicolumn{2}{c|}{$\| I_{h} u - u_{h}^{\infty} \|_{a}$} & $n_{0}$ \\
\hline
$0.4$ & $0.1026$ & & $0.3787$ & & $0.8338$ & & $0.2268$ & & $14$ \\
\hline
$0.2$ & $0.0312$ & $1.7$ & $0.1050$ & $1.9$ & $0.2483$ & $1.7$ & $0.0674$ & $1.8$ & $14$ \\
\hline
$0.1$ & $0.0094$ & $1.7$ & $0.0273$ & $1.9$ & $0.0771$ & $1.7$ & $0.0198$ & $1.8$ & $14$ \\
\hline
$0.05$ & $0.0020$ & $2.2$ & $0.0067$ & $2.0$ & $0.0245$ & $1.7$ & $0.0061$ & $1.7$ & $13$ \\
\hline
\end{tabular}
\end{center}
\caption{Convergence result on $\mathbb{CP}^{2}$ for $r=2$.}
\label{table:cp2-r2}
\end{table}

\begin{table}[ht]
\begin{center}
\begin{tabular}{|c|c|c|c|c|c|}
\hline
$h$ & $\|I_{h} u -  u_{h}^{n} \|_{L^{\infty}}$ & $\|I_{h} u -  u_{h}^{n} \|_{L^{2}}$ & $| I_{h} u - u_{h}^{n} |_{H^{1}}$ & $\| I_{h} u - u_{h}^{n} \|_{a}$ & $n$ \\
\hline
$0.24$ & $0.0691$ & $0.1832$ & $0.2147$ & $0.1332$ & $3$ \\
\hline
$0.12$ & $0.0175$ & $0.0480$ & $0.0462$ & $0.0321$ & $6$ \\
\hline
$0.06$ & $0.0043$ & $0.0127$ & $0.0132$ & $ 0.0087$ & $9$ \\
\hline
$0.03$ & $0.0012$ & $0.0034$ & $0.0040$ & $ 0.0025$ & $12$ \\
\hline
\end{tabular}
\end{center}
\caption{Convergence result on $\mathbb{CP}^{2}$ for $r=1.2$.}
\label{table:cp2-r1.2-tl}
\end{table}

\begin{table}[ht]
\begin{center}
\begin{tabular}{|c|c|c|c|c|c|}
\hline
$h$ & $\|I_{h} u -  u_{h}^{n} \|_{L^{\infty}}$ & $\|I_{h} u -  u_{h}^{n} \|_{L^{2}}$ & $| I_{h} u - u_{h}^{n} |_{H^{1}}$ & $\| I_{h} u - u_{h}^{n} \|_{a}$ & $n$ \\
\hline
$0.4$ & $0.1382$ & $0.9193$ & $0.8896$ & $0.2793$ & $2$ \\
\hline
$0.2$ & $0.0432$ & $0.2361$ & $0.2516$ & $0.0806$ & $3$ \\
\hline
$0.1$ & $0.0123$ & $0.0601$ & $0.0777$ & $0.0227$ & $4$ \\
\hline
$0.05$ & $0.0027$ & $0.0148$ & $0.0246$ & $0.0067$ & $5$ \\
\hline
\end{tabular}
\end{center}
\caption{Convergence result on $\mathbb{CP}^{2}$ for $r=2$.}
\label{table:cp2-r2-tl}
\end{table}

\subsection{A Problem on \texorpdfstring{$S^{2} \times S^{2}$}{S2xS2}}
Let $S^{2}$ be the unit sphere in $\mathbb{R}^{3}$, i.e.
\[
S^{2} = \left\{ (y_{1}, y_{2}, y_{3}) \in \mathbb{R}^{3} \middle| \sum_{i=1}^{3} y_{i}^{2}  = 1  \right\}.
\]
Let $M = S^{2} \times S^{2}$. Similar to $S^{4}$, we can decompose $S^{2}$ into two subdomains via stereographic projections. This decomposition results in a product decomposition of $S^{2} \times S^{2}$ with $2 \times 2 =4$ subdomains.

In the following, let $x= (x_{1}, x_{2})$ and $x' = (x'_{1}, x'_{2})$ denote the coordinates of $\mathbb{R}^{2}$. Let $\| x \| = \sqrt{\sum_{i=1}^{2} x_{i}^{2}}$ and $\| x' \| = \sqrt{\sum_{i=1}^{2} {x'}_{i}^{2}}$. For $1 \leq i \leq 4$, let
\[
D_{i} = [-r, r]^{4} = \{ (x, x') \mid x \in [-r, r]^{2}, x' \in [-r, r]^{2} \}.
\]
The product decomposition of $S^{2} \times S^{2}$ is given by diffeomorphisms
\begin{eqnarray*}
\phi_{1}: \ D_{1} & \rightarrow & M_{1} \subset S^{2} \times S^{2} \subset \mathbb{R}^{3} \times \mathbb{R}^{3} \\
(x, x') & \mapsto & \left( \frac{2x}{1 + \|x\|^{2}}, \frac{1 - \|x\|^{2}}{1 + \|x\|^{2}}, \frac{2x'}{1 + \|x'\|^{2}}, \frac{1 - \|x'\|^{2}}{1 + \|x'\|^{2}}  \right),
\end{eqnarray*}
\begin{eqnarray*}
\phi_{2}: \ D_{2} & \rightarrow & M_{2} \subset S^{2} \times S^{2} \subset \mathbb{R}^{3} \times \mathbb{R}^{3} \\
(x, x') & \mapsto & \left( \frac{2x}{1 + \|x\|^{2}}, \frac{1 - \|x\|^{2}}{1 + \|x\|^{2}}, \frac{2x'}{1 + \|x'\|^{2}}, \frac{-1 + \|x'\|^{2}}{1 + \|x'\|^{2}}  \right),
\end{eqnarray*}
\begin{eqnarray*}
\phi_{3}: \ D_{3} & \rightarrow & M_{3} \subset S^{2} \times S^{2} \subset \mathbb{R}^{3} \times \mathbb{R}^{3} \\
(x, x') & \mapsto & \left( \frac{2x}{1 + \|x\|^{2}}, \frac{-1 + \|x\|^{2}}{1 + \|x\|^{2}}, \frac{2x'}{1 + \|x'\|^{2}}, \frac{1 - \|x'\|^{2}}{1 + \|x'\|^{2}}  \right),
\end{eqnarray*}
and
\begin{eqnarray*}
\phi_{4}: \ D_{4} & \rightarrow & M_{4} \subset S^{2} \times S^{2} \subset \mathbb{R}^{3} \times \mathbb{R}^{3} \\
(x, x') & \mapsto & \left( \frac{2x}{1 + \|x\|^{2}}, \frac{-1 + \|x\|^{2}}{1 + \|x\|^{2}}, \frac{2x'}{1 + \|x'\|^{2}}, \frac{-1 + \|x'\|^{2}}{1 + \|x'\|^{2}}  \right).
\end{eqnarray*}
To guarantee $S^{2} \times S^{2} = \bigcup_{i=1}^{4} \mathrm{Int} M_{i}$, we have to let $r>1$. The larger $r$ is, the more overlapping there will be. The transitions of coordinates are given by
\[
\phi_{2}^{-1} \circ \phi_{1} (x,x') = \left( x, \frac{x'}{\| x' \|^{2}} \right), \qquad \phi_{3}^{-1} \circ \phi_{1} (x,x') = \left( \frac{x}{\| x \|^{2}}, x' \right),
\]
\[
\phi_{4}^{-1} \circ \phi_{1} (x,x') = \left( \frac{x}{\| x \|^{2}}, \frac{x'}{\| x' \|^{2}} \right), \qquad \phi_{3}^{-1} \circ \phi_{2} (x,x') = \left( \frac{x}{\| x \|^{2}}, \frac{x'}{\| x' \|^{2}} \right),
\]
\[
\phi_{4}^{-1} \circ \phi_{2} (x,x') = \left( \frac{x}{\| x \|^{2}}, x' \right), \qquad \phi_{4}^{-1} \circ \phi_{3} (x,x') = \left( x, \frac{x'}{\| x' \|^{2}} \right),
\]
and $\phi_{i}^{-1} \circ \phi_{j} = \phi_{j}^{-1} \circ \phi_{i}$ for all $i$ and $j$.

Equip $S^{2}$ with the Riemannian metric $g$ inherited from the standard one on $\mathbb{R}^{3}$. Equip $S^{2} \times S^{2}$ with the product metric. On each $D_{i}$, the metric has the form
\[
g = 4 (1+ \| x \|^{2})^{-2} \sum_{\alpha = 1}^{2} \mathrm{d} x_{\alpha} \otimes \mathrm{d} x_{\alpha} + 4 (1+ \| x' \|^{2})^{-2} \sum_{\alpha' = 1}^{2} \mathrm{d} x'_{\alpha'} \otimes \mathrm{d} x'_{\alpha'},
\]
and
\[
\Delta v = 4^{-1} (1+ \| x \|^{2})^{2} \sum_{\alpha=1}^{2} \frac{\partial^{2} v}{\partial x_{\alpha}^{2}} + 4^{-1} (1+ \| x' \|^{2})^{2} \sum_{\alpha' =1}^{2} \frac{\partial^{2} v}{\partial {x'}_{\alpha'}^{2}}.
\]

Consider the model problem \eqref{eqn_problem} on $S^{2} \times S^{2}$ with $b>0$. Choose the true solution $u$ to \eqref{eqn_problem} as
\[
u= y_{3} + y'_{3},
\]
where $S^{2} \times S^{2} \subset \mathbb{R}^{3} \times \mathbb{R}^{3}$, and $y_{3}$ (resp. $y'_{3}$) is the $3$rd coordinate of the first (resp. second) factor $\mathbb{R}^{3}$. Then
\[
f=(2+b) u
\]
in \eqref{eqn_problem}. On $D_{i}$, the true solution $u$ has the expression
\[
u \circ \phi_{1} = \frac{1 - \|x\|^{2}}{1 + \|x\|^{2}} + \frac{1 - \|x'\|^{2}}{1 + \|x'\|^{2}}, \qquad u \circ \phi_{2} = \frac{1 - \|x\|^{2}}{1 + \|x\|^{2}} + \frac{-1 + \|x'\|^{2}}{1 + \|x'\|^{2}},
\]
\[
u \circ \phi_{3} = \frac{-1 + \|x\|^{2}}{1 + \|x\|^{2}} + \frac{1 - \|x'\|^{2}}{1 + \|x'\|^{2}}, \qquad u \circ \phi_{4} = \frac{-1 + \|x\|^{2}}{1 + \|x\|^{2}} + \frac{-1 + \|x'\|^{2}}{1 + \|x'\|^{2}}.
\]
The weak form of \eqref{eqn_problem} on $D_{i}$ is formulated as: $\forall v \in H^{1}_{0} (D_{i})$,
\begin{eqnarray*}
&  & \int_{D_{i}} \left[ 4 (1 + \|x'\|^{2})^{-2} \sum_{\alpha =1}^{2} \frac{\partial u \circ \phi_{i}}{\partial x_{\alpha}} \frac{\partial v}{\partial x_{\alpha}} + 4 (1 + \|x\|^{2})^{-2} \sum_{\alpha' =1}^{2} \frac{\partial u \circ \phi_{i}}{\partial x'_{\alpha'}} \frac{\partial v}{\partial x'_{\alpha'}} \right. \\
&  & \left. + 16 (1 + \|x\|^{2})^{-2} (1 + \|x'\|^{2})^{-2} b \cdot u \circ \phi_{i} \cdot v \right] \mathrm{d} x_{1} \mathrm{d} x_{2} \mathrm{d} x'_{1} \mathrm{d} x'_{2} \\
& = & \int_{D_{i}} 16 (1 + \|x\|^{2})^{-2} (1 + \|x'\|^{2})^{-2} f \circ \phi_{i} \cdot v \mathrm{d} x_{1} \mathrm{d} x_{2} \mathrm{d} x'_{1} \mathrm{d} x'_{2}.
\end{eqnarray*}

Now we choose $b$ in \eqref{eqn_problem} as $2$. The numerical results are as follows in Tables \ref{table:s2s2-r1.2}, \ref{table:s2s2-r2}, \ref{table:s2s2-r1.2-tl} and \ref{table:s2s2-r2-tl}. The performance of our algorithm on $S^{2} \times S^{2}$ is similar to that on $S^{4}$ and $\mathbb{CP}^{2}$.

\begin{table}[ht]
\begin{center}
\begin{tabular}{|c|cc|cc|cc|cc|c|} \hline
$h$ & \multicolumn{2}{c|}{$\|I_{h} u -  u_{h}^{\infty} \|_{L^{\infty}}$} & \multicolumn{2}{c|}{$\|I_{h} u -  u_{h}^{\infty} \|_{L^{2}}$} & \multicolumn{2}{c|}{$| I_{h} u - u_{h}^{\infty} |_{H^{1}}$} & \multicolumn{2}{c|}{$\| I_{h} u - u_{h}^{\infty} \|_{a}$} & $n_{0}$ \\
\hline
$0.24$ & $0.0207$ & & $0.0588$ & & $0.1671$ & & $0.2175$ & & $22$ \\
\hline
$0.12$ & $0.0045$ & $2.2$ & $0.0144$ & $2.0$ & $0.0479$ & $1.8$ & $0.0606$ & $1.8$ & $22$ \\
\hline
$0.06$ & $0.0012$ & $1.9$ & $0.0036$ & $2.0$ & $0.0135$ & $1.8$ & $0.0164$ & $1.9$ & $22$ \\
\hline
$0.03$ & $3.1235e-4$ & $1.9$ & $8.6478e-4$ & $2.1$ & $0.0045$ & $1.6$ & $0.0053$ & $1.6$ & $21$ \\
\hline
\end{tabular}
\end{center}
\caption{Convergence result on $S^{2} \times S^{2}$ for $r=1.2$.}
\label{table:s2s2-r1.2}
\end{table}

\begin{table}[ht]
\begin{center}
\begin{tabular}{|c|cc|cc|cc|cc|c|} \hline
$h$ & \multicolumn{2}{c|}{$\|I_{h} u -  u_{h}^{\infty} \|_{L^{\infty}}$} & \multicolumn{2}{c|}{$\|I_{h} u -  u_{h}^{\infty} \|_{L^{2}}$} & \multicolumn{2}{c|}{$| I_{h} u - u_{h}^{\infty} |_{H^{1}}$} & \multicolumn{2}{c|}{$\| I_{h} u - u_{h}^{\infty} \|_{a}$} & $n_{0}$ \\
\hline
$0.4$ & $0.1452$ & & $0.9763$ & & $1.1952$ & & $1.0766$ & & $9$ \\
\hline
$0.2$ & $0.0234$ & $2.6$ & $0.1985$ & $2.3$ & $0.3646$ & $1.7$ & $0.3014$ & $1.8$ & $9$ \\
\hline
$0.1$ & $0.0090$ & $1.4$ & $0.0558$ & $1.8$ & $0.1176$ & $1.6$ & $0.0884$ & $1.8$ & $9$ \\
\hline
$0.05$ & $0.0016$ & $2.5$ & $0.0132$ & $2.1$ & $0.0374$ & $1.7$ & $0.0268$ & $1.7$ & $9$ \\
\hline
\end{tabular}
\end{center}
\caption{Convergence result on $S^{2} \times S^{2}$ for $r=2$.}
\label{table:s2s2-r2}
\end{table}

\begin{table}[ht]
\begin{center}
\begin{tabular}{|c|c|c|c|c|c|}
\hline
$h$ & $\|I_{h} u -  u_{h}^{n} \|_{L^{\infty}}$ & $\|I_{h} u -  u_{h}^{n} \|_{L^{2}}$ & $| I_{h} u - u_{h}^{n} |_{H^{1}}$ & $\| I_{h} u - u_{h}^{n} \|_{a}$ & $n$ \\
\hline
$0.24$ & $0.0334$ & $0.0975$ & $0.1543$ & $0.2640$ & $5$ \\
\hline
$0.12$ & $0.0063$ & $0.0215$ & $0.0448$ & $0.0687$ & $7$ \\
\hline
$0.06$ & $0.0022$ & $0.0079$ & $0.0128$ & $0.0214$ & $8$ \\
\hline
$0.03$ & $4.8270e-4$ & $0.0017$ & $0.0044$ & $0.0059$ & $10$ \\
\hline
\end{tabular}
\end{center}
\caption{Convergence result on $S^{2} \times S^{2}$ for $r=1.2$.}
\label{table:s2s2-r1.2-tl}
\end{table}

\begin{table}[ht]
\begin{center}
\begin{tabular}{|c|c|c|c|c|c|}
\hline
$h$ & $\|I_{h} u -  u_{h}^{n} \|_{L^{\infty}}$ & $\|I_{h} u -  u_{h}^{n} \|_{L^{2}}$ & $| I_{h} u - u_{h}^{n} |_{H^{1}}$ & $\| I_{h} u - u_{h}^{n} \|_{a}$ & $n$ \\
\hline
$0.4$ & $0.2436$ & $1.1576$ & $1.3937$ & $1.4708$ & $2$ \\
\hline
$0.2$ & $0.0296$ & $0.1829$ & $0.3708$ & $0.3186$ & $3$ \\
\hline
$0.1$ & $0.0088$ & $0.0546$ & $0.1175$ & $0.0893$ & $4$ \\
\hline
$0.05$ & $0.0021$ & $0.0120$ & $0.0375$ & $0.0281$ & $4$ \\
\hline
\end{tabular}
\end{center}
\caption{Convergence result on $S^{2} \times S^{2}$ for $r=2$.}
\label{table:s2s2-r2-tl}
\end{table}

\section*{Acknowledgements}
We thank Feng Wang, Yuanming Xiao and Xuejun Xu for various discussions. We thank the anonymous reviewers for many constructive
comments and suggestions which led to an improved presentation of this paper. S. Cao was partially supported by NSF awards DMS-2136075 and DMS-2309778. L. Qin was partially supported by NSFC11871272.


\end{document}